\documentclass[11pt]{article}

\usepackage{fullpage}

\usepackage{tikz-cd}
\usetikzlibrary{arrows}
\usepackage{enumitem}

\usepackage{amsthm}
\usepackage{amssymb}
\usepackage{amsmath}
\usepackage{amsxtra}
\usepackage{amsfonts}
\usepackage[all,2cell]{xy}
\usepackage{mathrsfs}
\usepackage{hyperref}
\hypersetup{
  bookmarksnumbered=true,
  colorlinks=true,
  linkcolor=blue,
  citecolor=blue,
  filecolor=blue,
  menucolor=blue,
  pagecolor=blue,
  urlcolor=blue,
  pdfnewwindow=true,
  pdfstartview=FitBH}
\usepackage{subfigure}
\usepackage{mathabx}
\usepackage{euscript}
\usepackage[bbgreekl]{mathbbol}
\usepackage{adjustbox}

\usepackage{graphics}

\usepackage[
backend=biber,
style=alphabetic,
]{biblatex}
\bibliography{refs}

\newsavebox{\pullback}
\sbox\pullback{%
\begin{tikzpicture}%
\draw (0,0) -- (1ex,0ex);%
\draw (1ex,0ex) -- (1ex,1ex);%
\end{tikzpicture}}

\newsavebox{\pushout}
\sbox\pushout{%
\begin{tikzpicture}%
\draw (1ex,0ex) -- (1ex,1ex);%
\draw (0ex,1ex) -- (1ex,1ex);%
\end{tikzpicture}}

\numberwithin{equation}{subsection}

\newtheoremstyle{customtheorem}
{10pt}                
{5pt}                
{}        
{}                
{\bfseries}       
{.}               
{ }               
{}                
\theoremstyle{customtheorem}

\newtheorem{thm}[equation]{Theorem}    
\newtheorem{lem}[equation]{Lemma}          
\newtheorem{cor}[equation]{Corollary} 
\newtheorem{prop}[equation]{Proposition}
\newtheorem{rem}[equation]{Remark}
\newtheorem{notat}[equation]{Notation}
\newtheorem{ex}[equation]{Example} 
\newtheorem{defn}[equation]{Definition}    

\setlist[enumerate,1]{label={(\arabic*)}}

\def\Ae{\EuScript{A}}
\def\Be{\EuScript{B}}
\def\Ce{\EuScript{C}}
\def\De{\EuScript{D}}

\def\Xe{\EuScript{X}}

\def\Le{\EuScript{L}}
\def\Ke{\EuScript{K}}
\def\He{\EuScript{H}}
\def\Me{\EuScript{M}}
\def\Se{\EuScript{S}}

\def\Oc{\mathcal{O}}

\def\Cc{\mathcal{C}}
\def\Fc{\mathcal{F}}
\def\Gc{\mathcal{G}}

\def\C{\mathbb{C}}

\def\Z{\mathbb{Z}}
\def\P{\mathbb{P}}

\title{Semiorthogonal decompositions of stable $\infty$-categories}
\author{Rio Haeussler Albi}

\begin{document}

\maketitle

\begin{abstract}
    We define semiorthogonal decompositions of stable $\infty$-categories of length $n$, extending the theory of semiorthogonal decompositions presented in \cite{dyck_spherical}. Prior to this, we explicitly construct an equivalence of $\infty$-categories relating Waldhausen diagrams to coherent complexes in stable $\infty$-categories. This allows us to view semiorthogonal decompositions from two different perspectives, each of which has its unique advantages and disadvantages. Under mild conditions, we prove a reconstruction theorem for semiorthogonal decompositions, recovering a stable $\infty$-category as the (op)lax limit of a diagram formed by the subcategories constituting its decomposition. We apply this reconstruction in the case of Beilinson's exceptional collection, to obtain a reconstruction of $D^b(\text{Coh}(\P^n))$.
\end{abstract}

\tableofcontents

\section*{Introduction}
\addcontentsline{toc}{section}{\protect\numberline{}Introduction}%

In algebraic geometry, derived categories of coherent sheaves, denoted typically by $D^b(\text{Coh}(X))$ for a noetherian scheme $X$, serve as fundamental invariants for the underlying scheme. \textit{Semiorthogonal decompositions} (also \textit{SODs}) of such derived categories were first introduced in \cite{bondal-kapranov} and provide a way to break them into simpler, more manageable pieces, thereby offering a powerful tool to gain a better understanding of the geometry of $X$.
The geometric significance of this formalism was then further demonstrated in \cite{bondal-orlov} by Bondal and Orlov, who analyzed how the theory behaves with respect to certain birational operations such as blow-ups, flips and flops. This line of research has since developed into a substantial program, surveyed in \cite{kuznetsov}.

However, the theory of semiorthogonal decompositions has classically taken place at the level of triangulated categories, an environment which is unfortunately burdened by technical defects such as the non-functoriality of the cone construction. This limitation has motivated the development of enhancements of triangulated categories, e.g. stable $\infty$-categories or dg-categories, which have emerged largely in the last two decades. At the price of a much more elaborate framework, we choose to use the language of \textit{stable $\infty$-categories} as our way to enhance triangulated categories, in order to gain access to the powerful tools of higher category theory, developed in \cite{lurie-htt, kerodon}. In particular, our perspective tremendously benefits from the universal properties that become available by adopting this environment. Yet, it is to be emphasized that the price we pay is not to be underestimated. Several results that are immediate at the level of triangulated categories require substantial work in the $\infty$-categorical setting: for instance, the equivalence between Waldhausen diagrams and coherent complexes (§\ref{subsec:waldhausen}.\ref{par:equiv}), which on homotopy categories is close to trivial, occupies several pages of carefully manipulating diagram categories.\\

The main goal of this paper is to extend the framework of semiorthogonal decompositions of stable $\infty$-categories of length $1$ developed in \cite{dyck_spherical} by considering semiorthogonal decompositions of any finite length. We note that, in the terminology of \cite{dyck_spherical}, \textit{coCartesian semiorthogonal decompositions} of stable $\infty$-categories were independently defined under the name \textit{recollement} in \cite[Appendix A.8]{lurie-ha} and for higher length in \cite{fiorenza}. Furthermore, we want to acknowledge that a much more general and sophisticated theory of so-called \textit{stratifications} over arbitrary posets was developed in \cite{stratified}. Our goal is to provide an approachable discussion of the more restrictive, but also more concrete situation of semiorthogonal decompositions, in the spirit of \cite[Ch. 2]{dyck_spherical}.\\

As an application of our results, we explain how to recover a version of a classical result of Beilinson, using completely different methods. Beilinson's exceptional collection
\[
\Oc_{\P^n}, \Oc_{\P^n}(1),\dots, \Oc_{\P^n}(n)
\]
on projective space $\P^n$, see \cite{beilinson}, yields a semiorthogonal decomposition of $D^b(\text{Coh}(\P^n))$. Beilinson used the exceptional collection to prove that the category $D^b(\text{Coh}(\P^n))$ is equivalent to the category $D^b(\text{Rep}_k(Q))$, where $Q$ denotes the quiver
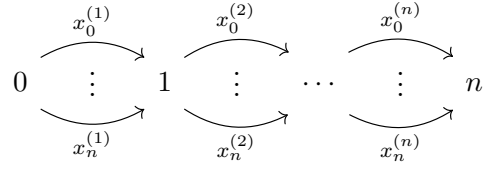
\[
\begin{tikzcd}[column sep = large]
	0 \ar[r,"x^{(1)}_0", shift left = 2, bend left] \ar[r, shift left = 1, draw=none, "\vdots" description] \ar[r,"x^{(1)}_n"' , shift right = 2, bend right] & 1 \ar[r,"x^{(2)}_0", shift left = 2, bend left] \ar[r, shift left = 1, draw=none, "\vdots" description] \ar[r,"x^{(2)}_n"' pos= 0.51 , shift right = 2, bend right] & \cdots \ar[r,"x^{(n)}_0", shift left = 2, bend left] \ar[r, shift left = 1, draw=none, "\vdots" description] \ar[r,"x^{(n)}_n"' pos= 0.49 , shift right = 2, bend right] & n
\end{tikzcd}
\]
with commutativity relations
\[
x_i^{(t+1)}x^{(t)}_j = x_j^{(t+1)}x^{(t)}_i
\]
for $0 \leq i,j \leq n$ and $1\leq t \leq n-1$. With our techniques, we reproduce this result in a slightly different form: For $n = 2$, our version of the result states that $D^b(\text{Coh}(\P^n))$ is obtained as the \textit{oplax limit} of the diagram
\[\begin{tikzcd}[column sep = small]
	&& {D^b(\text{Vect}_k^\text{fin})} \\
	{D^b(\text{Vect}_k^\text{fin})} && {} && {D^b(\text{Vect}_k^\text{fin}).}
	\arrow["{\oplus \{x,y,z\}}", from=1-3, to=2-5]
	\arrow["{\oplus \{x,y,z\}}", from=2-1, to=1-3]
	\arrow["{\oplus \{x^2,y^2,z^2,xy,xz,yz\}}"', from=2-1, to=2-5]
	\arrow[shorten <=3pt, shorten >=3pt, Rightarrow, from=2-3, to=1-3]
\end{tikzcd}
\]
Indeed, the latter is a very natural description of $D^b(\text{Rep}_k(Q))$: For instance, if we omit ``$D^b$", then the oplax limit of the resulting diagram is by definition the category $\text{Rep}_k(Q)$.\\
Hence, we expect our methods to yield similar gluing results in various examples.\\

The paper consists of two chapters.\\
In Chapter 1, we recall the relevant $\infty$-category theory and the basic theory of stable $\infty$-categories (§\ref{subsec:stable}). Moreover, we discuss locally (co)Cartesian fibrations and explain the Grothendieck construction, used to describe an $\infty$-functor implicitly in terms of a (co)Cartesian fibration (§\ref{groth-cons}). In particular, this enables us to define adjunctions of stable $\infty$-categories.

Chapter 2 constitutes the main part of the paper and is devoted to the definition and investigation of  semiorthogonal decompositions of length $n$. To obtain a formally clean definition, we utilize the $\infty$-categories of Waldhausen diagrams and coherent complexes of stable $\infty$-categories, which we prove to be equivalent  (§\ref{subsec:waldhausen}). We then proceed to define SODs and establish their elementary properties (§\ref{subsec:sod}). Furthermore, we discuss SODs which are \textit{locally (co)Cartesian}, a generalization of the notion of recollement, and characterize this property in terms of the existence of certain adjoints. In particular, we prove that such SODs satisfy a reconstruction/gluing theorem. Finally, as mentioned above, we apply our theory to Beilinson's exceptional collection on $\P^n$ (§\ref{subsec:beilinson}).

\subsection*{Acknowledgements}
I would like to deeply thank my supervisor Tobias Dyckerhoff for suggesting this project and for numerous helpful discussions throughout its development. Further, I am grateful to Muskan Abbas and Dhruv Kulshreshtha for valuable feedback. This work is based on the author's Bachelor's thesis.

\section{Preliminaries}

\subsection{(Stable) $\infty$-categories}\label{subsec:stable}

In this section, we briefly recall the basic language of $\infty$-categories (see \cite{joyal, lurie-htt}) and review the basic framework of stable $\infty$-categories (\cite{lurie-ha}), which we use throughout this paper.

\paragraph{Terminology.} We denote by $\text{Set}_\Delta$ the category of simplicial sets. For $n\geq 0$ and $0 \leq i \leq n$, we denote by $\Delta^n$ the standard $n$-simplex and by $\Lambda_i^n$ its $i$-th horn. The full subcategory of $\text{Set}_\Delta$ spanned by Kan complexes is denoted by $\text{Kan}$. Given a small category $\Cc$, we write $\text{N}\Cc$ to refer to its nerve and given a simplicial set $S$, we write $\text{h}S$ to refer to its homotopy category.

The nerve construction admits a generalization to \textit{simplicial categories}, i.e. categories enriched in $\text{Set}_\Delta$ (we denote by $\text{Set}_\Delta\text{-}\text{Cat}$ the category of small simplicial categories). Namely, one can define a functor
\[
\text{N}^\Delta : \text{Set}_\Delta\text{-}\text{Cat} \longrightarrow \text{Set}_\Delta
\]
called the \textit{simplicial nerve construction}, such that restriction to categories enriched in discrete simplicial sets yields the ordinary nerve functor. This functor further carries simplicial categories $C$ such that $\text{Hom}_C(X,Y)$ is a Kan complex for any $X,Y\in C$, to $\infty$-categories. As the specific construction will not play a role in this work, we refer to \cite[§1.1.5]{lurie-htt} for more details.

\paragraph{$\infty$-categories.}\label{infty-cat} Throughout this work, we will mean by an \textit{$\infty$-category} $\EuScript C$ a \textit{quasicategory} in the sense of \cite{joyal}, i.e. a simplicial set that satisfies the \textit{inner horn filling condition}: For $n> 0$ and $0 < i <n$, every morphism of simplicial sets $\Lambda_i^n \rightarrow \EuScript C$ can be extended to an $n$-simplex $\Delta^n \rightarrow \EuScript C$. Note that Kan complexes (also called \textit{$\infty$-groupoids}) and nerves of ordinary categories are $\infty$-categories.

A convenient feature of this approach to $\infty$-categories is that a \textit{functor between $\infty$-categories $\EuScript C, \EuScript D$} (also \textit{$\infty$-functor}) is simply a morphism of simplicial sets $f : \EuScript C \rightarrow \EuScript D$. Functors between $\infty$-categories $\Ce,\De$ naturally form an $\infty$-category $\text{Fun}(\Ce,\De)$ (see \cite[§2.2.5]{lurie-htt}). Commonly, we will consider functors of the form $\text{N}\mathcal I \rightarrow \Ce$, for an ordinary category $\mathcal I$. In this case, we simply write $\text{Fun}(\mathcal I,\EuScript D)$ to refer to the $\infty$-category of such functors.

Let $\Ce$ be an $\infty$-category. For $x,y \in \Ce$, there is a simplicial set $\text{Map}_\Ce(x,y)$ called the \textit{mapping space from $x$ to $y$}, defined via the iterated fiber product
\[
\{x\} \times_{\text{Fun}(\{0\},\EuScript C)} \text{Fun}(\Delta^1, \EuScript C) \times_{\text{Fun}(\{1\},\EuScript C)} \{y\}
\]
formed in the category of simplicial sets. It can be shown that $\text{Map}_\EuScript C(x,y)$ is a Kan complex, which justifies thinking of this construction as a space. While there is no preferred choice of a composition law
\[
\circ : \text{Map}_\EuScript C(y,z) \times \text{Map}_\EuScript C(x,y) \longrightarrow \text{Map}_\EuScript C(x,z),
\]
it is well-defined up to homotopy, i.e. as a morphism in the homotopy category $\text{hKan}$.

\paragraph{Limits and cofinality.}
Let $\Ce$ be an $\infty$-category. An object $x \in \EuScript C$ is \textit{initial} (resp. \textit{final}) if for all objects $y \in \EuScript C$, the mapping space $\text{Map}_\EuScript C(x,y)$ (resp. $\text{Map}_\EuScript C(y,x)$) is a contractible Kan complex. Using the $\infty$-categorical analogues of over- and undercategories, all universal constructions from ordinary category theory generalize to $\infty$-categories. The following constructions will appear throughout this paper:
\begin{itemize}
    \item Limits and colimits (see \cite[§1.2.13]{lurie-htt}).
    \item Left and right Kan extensions (see \cite[§4.3]{lurie-htt}).
    \item Left and right adjoints (see \cite[§5.2.2]{lurie-htt}).
\end{itemize}

A useful tool for computing limits and colimits is the notion of cofinality. A morphism of simplicial sets $f: A \rightarrow \Ce$ will be called \textit{left cofinal} (resp. \textit{right cofinal}) if for every object $x$ of $\Ce$, the simplicial set $A \times_\Ce \Ce_{/x}$ (resp. $A \times_\Ce \Ce_{x/}$) is weakly contractible.

\begin{lem}{\cite[\href{https://kerodon.net/tag/02XR}{Tag 02XR}]{kerodon}}
    Let $\Ce$ be an $\infty$-category and $e: A \to B$ a left cofinal morphism. Then a morphism $F : B^\lhd \to \Ce$ is a limit diagram if and only if
    \[
    A^\lhd \overset{e^\lhd}\longrightarrow B^\lhd \overset{F}\longrightarrow \Ce
    \]
    is a limit diagram. \qed
\end{lem}

\paragraph{Relative Kan extensions.}
The following technical lemma will be frequently used in proofs.
\begin{lem}{\cite[Prop. 4.3.2.15]{lurie-htt}}\label{lem:relative}
	Suppose we are given a diagram of $\infty$-categories
	\[
	\Ce \rightarrow \De' \overset{p}\leftarrow \De,
	\]
	where $p$ is a categorical fibration. Let $\Ce^0$ be a full subcategory of $\Ce$. Let $\text{Fun}'_{\De'}(\Ce,\De)\subset \text{Fun}_{\De'}(\Ce,\De)$ be the full subcategory spanned by functors $F$ which are $p$-left Kan extensions of $F\vert \Ce^0$ and let $\text{Fun}'_{\De'}(\Ce^0,\De) \subset \text{Fun}_{\De'}(\Ce^0,\De)$ be the full subcategory spanned by functors $F$ such that for each $C \in \Ce$, the induced diagram 
	$\Ce^0_{/C} \to \De$
	has a $p$-colimit. Then the restriction functor
	\[
	\text{Fun}'_{\De'}(\Ce,\De) \longrightarrow \text{Fun}'_{\De'}(\Ce^0,\De)
	\]
	is a trivial Kan fibration.
\end{lem}

By passing to opposite categories, we obtain a dual statement, where $p$-left Kan extensions, overcategories and $p$-colimits are replaced by $p$-right Kan extensions, undercategories and $p$-limits. In most cases, we will only need the following special case for $p : \De \to \Delta^0$.
\begin{lem}\label{lem:2}
    Let $\EuScript C, \EuScript D$ be $\infty$-categories and $\EuScript C_0 \subset \EuScript C$ a full subcategory. Further, let $\text{Fun}'(\EuScript C,\EuScript D) \subset \text{Fun}(\EuScript C,\EuScript D)$ be the full subcategory spanned by functors $F$ which are left Kan extensions of $F\vert_{\EuScript C_0}$ and let $\text{Fun}'(\EuScript C_0,\EuScript D) \subset \text{Fun}(\EuScript C_0,\EuScript D)$ be the full subcategory spanned by functors $F$ such that for every $C \in \EuScript C_0$, the induced diagram
    \[
    {{\EuScript C}_0}_{/C} := \EuScript C_{/C} \times_{\EuScript C} \EuScript C_0 \longrightarrow \EuScript D
    \]
    has a colimit. Then, restriction to $\EuScript C_0$ induces a trivial Kan fibration
    \[
    \pushQED{\qed}
    \text{Fun}'(\EuScript C,\EuScript D) \longrightarrow \text{Fun}'(\EuScript C_0,\EuScript D).\qedhere
    \popQED
    \]
\end{lem}

Note that if $\EuScript C, \EuScript C_0$ are the nerves of finite posets and $\EuScript D$ is a stable $\infty$-category (see \ref{subsec:stable}\ref{par:stab}), then the colimits always exist so that
\[
\text{Fun}'(\EuScript C_0,\EuScript D) = \text{Fun}(\EuScript C_0,\EuScript D).
\]

\paragraph{The $\infty$-category of $\infty$-categories.} We let $\EuScript Cat_\infty^\simeq$ denote the subcategory of $\text{Set}_\Delta$, where the objects are given by the collection of $\infty$-categories and for $\infty$-categories $\EuScript C, \EuScript D$, the simplicial set $\text{Hom}_{\EuScript Cat_\infty^\simeq} (\EuScript C,\EuScript D)$ is defined to be the \textit{largest Kan complex contained in} $\text{Fun}(\EuScript C,\EuScript D)$, see \cite[Def. 3.0.0.1]{lurie-htt}. Then, $\EuScript Cat_\infty := \text{N}^\Delta(\EuScript Cat_\infty^\simeq)$ defines an $\infty$-category called the \textit{$\infty$-category of $\infty$-categories}. We will sometimes make reference to the $(\infty,2)$-category $\C at_\infty$ of $\infty$-categories, although we do not define this in any formal sense.

\paragraph{Stable $\infty$-categories.}\label{par:stab} Let $\EuScript C$ be an $\infty$-category. We say that $\EuScript C$ is \textit{pointed} if it has a zero object. In this situation, a commutative square $\Delta^1\times \Delta^1 \rightarrow \EuScript C$ of the form
\[
\begin{tikzcd}
    x \ar[r, "f"]\ar[d] & y \ar[d, "g"]\\
    0 \ar[r] & z
\end{tikzcd}
\]
is called a \textit{fiber sequence} if it is a pullback square (in the $\infty$-categorical sense), and a \textit{cofiber sequence} if it is a pushout square.

An $\infty$-category $\EuScript C$ is called \textit{stable} if it satisfies the following conditions:
\begin{enumerate}[label=(\arabic*)]
    \item $\EuScript C$ is pointed.
    \item Every morphism in $\EuScript C$ admits a fiber and a cofiber.
    \item A commutative square $\Delta^1 \times \Delta^1 \rightarrow \EuScript C$ is a fiber sequence if and only if it is a cofiber sequence.
\end{enumerate}
Recall that under these conditions, fibers and cofibers arrange into functors
\begin{gather*}
\text{Fib}: \text{Fun}(\Delta^1, \Ce) \longrightarrow \Ce,\\
\text{Cof}: \text{Fun}(\Delta^1, \Ce) \longrightarrow \Ce.
\end{gather*}

The following is a well-known and convenient characterization of stable $\infty$-categories.

\begin{prop}{\cite[Prop. 1.1.3.4]{lurie-ha}}
    An $\infty$-category $\EuScript C$ is stable if and only if it satisfies the following conditions:
    \begin{enumerate}[label=(\arabic*)]
        \item $\EuScript C$ is pointed.
        \item $\EuScript C$ has finite limits and colimits.
        \item A commutative square $\Delta^1 \times \Delta^1 \rightarrow \EuScript C$ is \textit{Cartesian} (i.e. a pullback square) if and only if it is \textit{coCartesian} (i.e. a pushout square). Such squares are called \textit{biCartesian}.\pushQED{\qed} \qedhere \popQED
    \end{enumerate}
\end{prop}

A functor $f : \EuScript C \rightarrow \EuScript D$ between stable $\infty$-categories is \textit{exact} if it preserves zero objects and fiber/cofiber sequences. Equivalently, a functor is exact if it is either left exact (preserves finite limits) or right exact (preserves finite colimits), see \cite[Prop. 1.1.4.1]{lurie-ha}. We define the \textit{$\infty$-category of stable $\infty$-categories} to be the simplicial subset $\EuScript{S}t \subset \EuScript {C}at_\infty$ formed by simplices whose vertices correspond to stable $\infty$-categories and edges to exact functors.

Since pushout and pullback squares coincide in stable $\infty$-categories, we have the following ``combined" version of the pasting lemma:
\begin{lem}{\cite[Prop. 4.4.2.1]{lurie-htt}}\label{lem:5}
    Let $\EuScript C$ be a stable $\infty$-category and suppose we are given a diagram $\Delta^2 \times \Delta^1 \rightarrow \EuScript C$ of the form:
    \[
    \begin{tikzcd}
        x \ar[r] \ar[d] & y \ar[r]\ar[d] & z \ar[d]\\
        x\ar[r] & y' \ar[r] & z'
    \end{tikzcd}
    \]
    \begin{enumerate}[label=(\alph*)]
        \item If the left square is biCartesian, then the right square is biCartesian if and only if the outer square is biCartesian.
        \item If the right square is biCartesian, then the left square is biCartesian if and only if the outer square is biCartesian. \pushQED{\qed} \qedhere \popQED
    \end{enumerate}
\end{lem}

\subsection{Locally (co)Cartesian fibrations and the Grothendieck construction}\label{groth-cons}

We introduce the concept of (co)Cartesian fibrations and locally (co)Cartesian fibrations and briefly discuss the $\infty$-categorical Grothendieck construction (see \cite[§2.4]{lurie-htt}) which establishes an equivalence between Cartesian fibrations and functors into $\Ce at_\infty$.

\paragraph{Locally Cartesian and coCartesian fibrations.}

\begin{defn}\label{defn:1}
    Let $p : X \rightarrow S$ be an inner fibration of simplicial sets.
    \begin{enumerate}[label=(\alph*)]
        \item An edge $e : x \rightarrow y$ of $X$ is called \textit{$p$-Cartesian} if the induced map
        \[
        X_{/e} \longrightarrow X_{/y} \times_{S_{/p(y)}} S_{/p(f)}
        \]
        is a trivial Kan fibration.
        \item Dually, An edge $e : x \rightarrow y$ of $X$ is called \textit{$p$-coCartesian} if the induced map
        \[
        X_{e/} \longrightarrow X_{x/} \times_{S_{p(x)/}} S_{p(f)/}
        \]
        is a trivial Kan fibration.
    \end{enumerate}
    The inner fibration $p$ is called a \textit{Cartesian fibration} if for every edge $e : x \rightarrow y$ of $S$ and every $\widetilde{y} \in X$ with $p(\widetilde{y}) = y$, there exists a $p$-Cartesian edge $\widetilde{e} : \widetilde{x} \rightarrow \widetilde{y}$ such that $p(\widetilde{e}) = e$. Dually, it is called a \textit{coCartesian fibration} if for every edge $e : x \rightarrow y$ of $S$ and every $\widetilde{x} \in X$ with $p(\widetilde{x}) = x$, there exists a $p$-coCartesian edge $\widetilde{e} : \widetilde{x} \rightarrow \widetilde{y}$ such that $p(\widetilde{e}) = e$.
\end{defn}

\begin{defn}\label{defn:2}
    Let $p : X \rightarrow S$ be an inner fibration of simplicial sets. Then $p$ is called \textit{locally Cartesian} (resp. \textit{locally coCartesian}) if for every edge $e : \Delta^1 \rightarrow S$, the pullback
    \[
    p' : X \times_S \Delta^1 \longrightarrow \Delta^1
    \]
    is a Cartesian (resp. coCartesian) fibration.
\end{defn}

\begin{rem}
For our purposes, $p : \EuScript X \rightarrow \EuScript S$ will be a functor of $\infty$-categories and $\EuScript S$ will be the nerve of a category. In particular, $p$ is automatically an inner fibration (\cite[Prop. 2.3.1.5]{lurie-htt}).
\end{rem}

\begin{ex}\label{ex:1} Let $\EuScript X$ be an $\infty$-category.
\begin{enumerate}[label=(\arabic*)]
\item The projection $p : X \rightarrow \Delta^0$ is both Cartesian and coCartesian. Namely, for each object $x \in \EuScript X$ the identity on $x$ is both $p$-Cartesian and $p$-coCartesian. More generally, any equivalence of $\EuScript X$ is both $p$-Cartesian and $p$-coCartesian.
\item A functor $p : X \rightarrow \Delta^1$ is locally Cartesian (resp. locally coCartesian) if and only if it is Cartesian (resp. coCartesian). This follows from the fact that for the unique nondegenerate edge $\text{id}_{\Delta^1}$ of $\Delta^1$, the pullback map
\[
p' : X \times_{\Delta^1} \Delta^1 \longrightarrow \Delta^1
\]
is isomorphic as a map of simplicial sets to $p$.
\end{enumerate}
\end{ex}

\paragraph{The Grothendieck construction and adjunctions of $\infty$-categories.}

Let $\mathcal C$ be an ordinary category and $q : \mathcal C \rightarrow \text{Set}_\Delta$ a functor such that for every $c \in \mathcal C$, $q(c)$ is an $\infty$-category.

\begin{defn}{\cite[Def. 1.1.9]{dyck_spherical}}
	The \textit{covariant Grothendieck construction} of $q$ is the simplicial set $\Gamma(q)$ defined as follows. An $n$-simplex of $\Gamma(q)$ consists of
	\begin{enumerate}
		\item a functor $\sigma : [n] \rightarrow \mathcal C$, 						corresponding to an $n$-simplex of $\text{N}(\mathcal C)$,
		\item for every $I \subset [n]$, a map $\Delta^I \rightarrow 				q(\sigma(\max(I)))$,
	\end{enumerate}
	such that for every $J \subset I \subset [n]$, the diagram
	\[
	\begin{tikzcd}
		\Delta^J \ar[r] \ar[d] & {q(\sigma(\max(J)))} \ar[d]\\
		\Delta^I \ar[r] & {q(\sigma(\max(I)))}
	\end{tikzcd}
	\]
	commutes. There is a natural forgetful $\infty$-functor $\pi : \Gamma(q) \rightarrow \text{N}(\mathcal C)$.
	Dually, the \textit{contravariant Grothendieck construction} of $q$ is the simplicial set $\chi(q)$ defined as follows. An $n$-simplex of $\chi(q)$ consists of
	\begin{enumerate}
		\item a functor $\sigma : [n] \rightarrow \mathcal C^\text{op}$, 			corresponding to an $n$-simplex of $\text{N}(\mathcal C)^\text{op}			$,
		\item for every $I \subset [n]$, a map $\Delta^I \rightarrow 				q(\sigma(\min(I)))$,
	\end{enumerate}
	such that for every $J \subset I \subset [n]$, the diagram
	\[
	\begin{tikzcd}
		\Delta^J \ar[r] \ar[d] & {q(\sigma(\min(J)))} \ar[d]\\
		\Delta^I \ar[r] & {q(\sigma(\min(I)))}
	\end{tikzcd}
	\]
	commutes. There is a natural forgetful $\infty$-functor $\pi' : \chi(q) \rightarrow \text{N}(\mathcal C)^\text{op}$.
\end{defn}
There is a version of this construction taking as its input the more general datum of an $\infty$-functor $\text{N}(\mathcal C) \rightarrow \Ce at_\infty$, see \cite[§3.2]{lurie-htt}. Due to its technicality, we do not give a description here.

We assume the following without proof.

\begin{prop}{\cite[Prop. 3.2.5.21]{lurie-htt}}
\begin{enumerate}[label=(\alph*)]
\item The map $\pi : \Gamma(q) \rightarrow \text{N}(\mathcal C)$ is a coCartesian fibration and the map $\pi' : \chi(q) \rightarrow \text{N}(\mathcal C)^\text{op}$ is a Cartesian fibration.
\item Let $e$ be an edge of $\Gamma(q)$ covering a morphism $c \rightarrow c'$ of $\mathcal C$. Then $e$ is $\pi$-coCartesian if and only if the corresponding edge of $q(c')$ is an equivalence.
\item Let $e$ be an edge of $\chi(q)$ covering a morphism $c \rightarrow c'$ of $\mathcal C^\text{op}$. Then $e$ is $\pi'$-Cartesian if and only if the corresponding edge of $q(c)$ is an equivalence. \pushQED{\qed} \qedhere \popQED
\end{enumerate}
\end{prop}

Conversely, given a coCartesian fibration (resp. Cartesian fibration) $p : \Xe \rightarrow \text{N}(\mathcal C)$, there exists a \textit{classifying $\infty$-functor} $Q : \text{N}(\mathcal C) \rightarrow \Ce at_\infty$ (resp. $Q : \text{N}(\mathcal C)^\text{op} \rightarrow \Ce at_\infty$), i.e. a functor $Q$ such that the covariant (resp. contravariant) Grothendieck construction of $Q$ reproduces $p$. Moreover, $Q$ is unique up to a contractible space of choices and for $c \in \mathcal C$, we have $Q(c) = p^{-1}(c)$ (which is an $\infty$-category since $p$ is an inner fibration).

\begin{ex}\label{ex:2}
Suppose that $\mathcal C = [1]$ and let $p : \Xe \rightarrow \Delta^1$ be a coCartesian fibration. In this case, the datum of an $\infty$-functor $\Delta^1 \rightarrow \Ce at_\infty$ corresponds exactly to the datum of an ordinary functor $q : [1] \rightarrow \text{Set}_\Delta$ such that $q(0),q(1)$ are $\infty$-categories. Namely, they both amount to a single $\infty$-functor $F : \Ae \rightarrow \Be$. We thus have the following:
\begin{enumerate}[label=(\alph*)]
	\item There exists an $\infty$-functor $F : p^{-1}(0) \rightarrow p^{-1}(1)$ and an isomorphism $\Gamma(F) \cong \Xe$ such that $\pi : \Gamma(F) \rightarrow \Delta^1$ is identified with $p$. Furthermore, $F$ is unique up to contractible choice.
	\item Let $a \to b$ be an edge of $\Xe$ with $a \in p^{-1}(0), b \in p^{-1}(1)$, which by the isomorphism $\Gamma(F) \cong \Xe$, corresponds uniquely to a morphism $F(a) \to b$. Then $a \to b$ is $p$-coCartesian if and only if $F(a) \to b$ is an equivalence. In particular, there exists a $p$-coCartesian edge $a \to F(a)$.
\end{enumerate} 
\end{ex}

\begin{rem}\label{rem:straightening-locally}
    Another variant of the Grothendieck construction, known as straightening for locally coCartesian fibrations (see \cite[Thm. 3.8.1]{lurie-goodwillie}), asserts a correspondence between locally coCartesian $\chi \to \Ce$ (resp. Cartesian) fibrations and oplax functors $\Ce \to \Ce at_\infty$ of $(\infty,2)$-categories. 
\end{rem}

Using the Grothendieck construction, we can define the notion of an \textit{adjunction} of $\infty$-categories.
\begin{defn}
	\begin{enumerate}[label=(\alph*)]
		\item An \textit{adjunction} of $\infty$-categories is a map
		\[
		p : \Xe \longrightarrow \Delta^1
		\]
		that is a \textit{biCartesian fibration}, i.e. both a Cartesian and a coCartesian fibration.
		\item Let $F : \Ae \rightarrow \Be$ be an $\infty$-functor. We say that
		\begin{enumerate}[label=(\roman*)]
			\item $F$ \textit{admits a right adjoint} if the covariant Grothendieck construction $\pi: \Gamma(F) \rightarrow \Delta^1$ is also a Cartesian fibration. If $G : \Be \to \Ae$ classifies the Cartesian fibration $\pi$, then we say that $G$ \textit{is left adjoint to} $F$ and $F$ \textit{is right adjoint to} $G$.
			 \item $F$ \textit{admits a left adjoint} if the contravariant 	Grothendieck construction $\pi' : \chi(F) \rightarrow \Delta^1$ is also a coCartesian fibration. If $G : \Be \to \Ae$ classifies the coCartesian fibration $\pi'$, then we say that $G$ \textit{is right adjoint to} $F$ and $F$ \textit{is left adjoint to} $G$.
		\end{enumerate}
	\end{enumerate}
\end{defn}

As in ordinary category theory, adjoints satisfy exactness conditions.

\begin{prop}{\cite[Prop. 5.2.3.5]{lurie-htt}}
	Let $F : \Ae \to \Be$ be an $\infty$-functor. If $F$ has a right adjoint, then it preserves all colimits which exist in $\Ae$. If $F$ has a left adjoint, then it preserves all limits which exist in $\Ae$.
\end{prop}

In the case of stable $\infty$-categories, we obtain a stronger statement.

\begin{cor}\label{cor:adjunction}
Let $F : \Ae \to \Be$ be an $\infty$-functor between stable $\infty$-categories. If $F$ has either a right or a left adjoint, then it is exact.
\end{cor}
\begin{proof}
If $F$ has a right adjoint, then it preserves all colimits existing in $\Ae$. In particular, it preserves all finite colimits and is thus right exact. However, between stable $\infty$-categories this is equivalent to $F$ being exact (see §\ref{subsec:stable}.\ref{par:stab}). If $F$ has a left adjoint, then it is left exact and therefore exact.
\end{proof}

\paragraph{Oplax limits.}\label{par:oplax} Let $p : \chi \to \Ce$ be a locally coCartesian fibration where $\Ce$ is an $\infty$-category. By virtue of Lurie's Straightening equivalence, this corresponds to an oplax diagram in the $(\infty,2)$-category of $\infty$-categories, see Remark \ref{rem:straightening-locally}. Naturally associated to this datum are four universal constructions: the \textit{oplax}/\textit{lax} \textit{limit}/\textit{colimit}. In the following, we explicitly describe the oplax limit associated to a locally coCartesian fibration $p : \chi \to \Delta^n$, circumventing the Straightening equivalence.

\begin{notat}
    Fix a poset $P$ and set $\Delta' := \Delta \cup \{\varnothing\}$. We define the category of \textit{subdivisions of $P$} to be the full subcategory \[\text{sd}(P) \subset \Delta'_{/P} := \Delta' \times_{\text{Cat}} \{P\}\] consisting of injective functors. Notice that this again defines a poset.
\end{notat}

\begin{ex}\label{ex:sdcube}
    The subdivisions of $[n]$ form an $(n+1)$-dimensional cube: We have a map
    \[
    \varphi : \text{sd}([n]) \longrightarrow [1]^{n+1}, \quad (f : [i] \to [n]) \longmapsto \varphi(f),
    \]
    where $\varphi(f) = (\varphi(f)_0,\dots,\varphi(f)_n)$ is given by
    \[
    \varphi(f)_k = \begin{cases}
        1 &\text{if }k = f(j) \text{ for some }j \in [i],\\
        0 &\text{else.}
    \end{cases}
    \]
    It is an easy exercise to show that this is order-preserving and bijective.
Furthermore, note that subdivisions $f \in \text{sd}([n])$ correspond to strictly increasing sequences $n_0n_1\dots n_i$ of integers $n_k \in [n]$. In practice, we will use this as our notation.
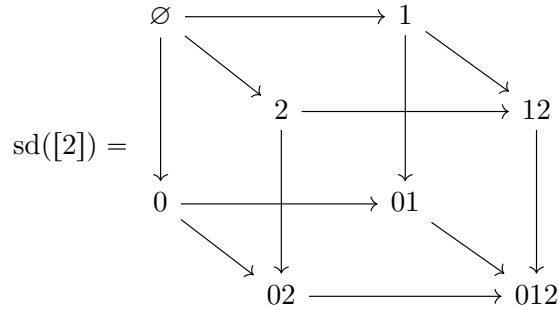
\[\text{sd}([2]) = \begin{tikzcd}
	\varnothing && 1 \\
	& 2 && 12 \\
	0 && 01 \\
	& 02 && 012
	\arrow[from=1-1, to=1-3]
	\arrow[from=1-1, to=2-2]
	\arrow[from=1-1, to=3-1]
	\arrow[from=1-3, to=2-4]
	\arrow[from=1-3, to=3-3]
	\arrow[from=2-2, to=2-4]
	\arrow[from=2-2, to=4-2]
	\arrow[from=2-4, to=4-4]
	\arrow[from=3-1, to=3-3]
	\arrow[from=3-1, to=4-2]
	\arrow[from=3-3, to=4-4]
	\arrow[from=4-2, to=4-4]
\end{tikzcd}\]
\end{ex}

\begin{defn}\label{def:oplax-lim}
    Let $p : \chi \to \Delta^n$ be a locally coCartesian fibration and define a map
\begin{align*}
\text{sd}([n])\setminus \varnothing \longrightarrow \Delta^n, \quad n_0n_1\dots n_i \longmapsto n_i.
\end{align*} Then, the $\infty$-category
    \[
    \text{oplax-lim}(p) \subset \text{Fun}_{\Delta^n}(\text{sd}([n])\setminus \varnothing, \chi)
    \]
    is defined as the full subcategory consisting of functors $L$ such that for every $0 \leq i \leq n$ and a subdivision $n_0n_1\dots n_i$, the edge $L(n_0n_1\dots n_{i-1} \to n_0n_1\dots n_{i})$ is locally $p$-coCartesian. We call $\text{oplax-lim}(p)$ the \textit{oplax limit} of $p$.
\end{defn}

\begin{rem}
    Note that we have functors
    \[
    \varphi_i : \text{oplax-lim}(p) \longrightarrow p^{-1}(i)
    \]
    given by evaluation at the subdivision $i$. If for $i < j$, $f_{ij} : p^{-1}(i) \to p^{-1}(j)$ denotes the functor associated to the coCartesian fibration $p\vert_{p^{-1}(\Delta^{\{i,j\}})}$ via the Grothendieck construction, then there is a coherent system of natural transformations
    \[
    \varphi_j \Longrightarrow f_{ij} \circ \varphi_i.
    \]
    Since a formal statement or proof of this would be somewhat technical and is not important for our discussion, we omit the details here.
\end{rem}

\section{Semiorthogonal decompositions}

\subsection{Waldhausen diagrams and coherent complexes}\label{subsec:waldhausen} In this section, we construct an $\infty$-categorical equivalence, which will be useful in the subsequent discussion of semiorthogonal decompositions.

\paragraph{The Waldhausen S-construction.} First, we briefly recall the Waldhausen S-construction of a stable $\infty$-category, as defined in \cite[§3.1]{dyck_spherical}. 

\begin{notat}
For $n \geq 0$, we denote by
\[
J(n) := \{(i,j) \mid 0 \leq i \leq j \leq n\}
\]
the poset with order $(i,j) \leq (k,l)$ iff $i \leq k$ and $j \leq l$. We consider $J(n)$ as a category. Note that $J(n) = \text{Mor}[n]$ is the category of morphisms of the poset $[n]$ considered as a category. The assignment $[n] \mapsto J(n)$ defines a covariant functor $J : \Delta \rightarrow \EuScript{C}at$.
\end{notat}

\begin{defn}{\cite[Def. 3.1.2]{dyck_spherical}}
    Let $\EuScript C$ be a stable $\infty$-category. We denote by \[S_n(\EuScript C) \subset \text{Fun}(J(n), \EuScript C)\] the full subcategory spanned by functors $X : J(n) \rightarrow \EuScript C$ satisfying the conditions:
    \begin{enumerate}[label=(\arabic*)]
        \item For each $0 \leq i \leq n$ we have $X(i,i) = 0$.
        \item For each $0 \leq i < k \leq j < l \leq n$ the square
        \[
        \begin{tikzcd}
            X(i,j) \arrow[r] \arrow[d] & X(i,l) \arrow[d]\\
            X(k,j) \arrow[r] & X(k,l)
        \end{tikzcd}
        \]
        is a biCartesian square in $\EuScript C$.
    \end{enumerate}
    The functor $X$ is called a \textit{Waldhausen diagram of length $n$} and we abbreviate $X_{ij} := X(i,j)$. When $n$ varies, the $\infty$-categories $S_n(\EuScript C)$ unite into a simplicial $\infty$-category $S_\bullet(\EuScript C)$ called the \textit{Waldhausen S-construction of $\EuScript C$}.
\end{defn}

\begin{ex}
    Note that $S_0(\Ce)$ is simply the full subcategory of $\Ce$ spanned by the zero objects. Moreover, $S_1(\Ce)$ consists of diagrams of the form
    \[
    \begin{tikzcd}
        0 \arrow[r] & X \arrow[d]\\
        & 0
    \end{tikzcd}
    \]
    and is thus equivalent to the stable $\infty$-category $\Ce$ itself. Similarly, $S_2(\Ce)$ is equivalent to the $\infty$-category of biCartesian squares
    \[
    \begin{tikzcd}
        X \arrow[r] \arrow[d] & Y \arrow[d]\\
        0 \arrow[r] & Z
    \end{tikzcd}
    \]
    i.e., cofiber/fiber sequences. In general, we omit the zero objects $X_{00}$ and $X_{11}$ so that objects of $S_n(\Ce)$ correspond to diagrams
    \[
    \begin{tikzcd}
    X_{01} \arrow[r] \arrow[d] &[10pt] X_{02} \arrow[r] \arrow[d] &[6pt] X_{03} \arrow[r] \arrow[d] &[-5pt] \cdots \arrow[r] \arrow[d] &[-12pt] X_{0,n} \arrow[d]\\
    0 \arrow[r] & X_{12} \arrow[r] \arrow[d] & X_{13} \arrow[r] \arrow[d] & \ddots \arrow[r] \arrow[d] & \vdots \arrow[d] \\[1pt]
    & 0 \arrow[r] & \ddots \arrow[r]\arrow[d] & X_{n-3,n-1} \arrow[r] \arrow[d] & X_{n-3,n} \arrow[d]\\[9pt]
    & & 0 \arrow[r] & X_{n-2,n-1} \arrow[r] \arrow[d] &X_{n-2,n} \arrow[d]\\[9pt]
    &&& 0 \arrow[r] & X_{n-1,n}
    \end{tikzcd}
    \]
    where all squares are biCartesian. Note that restriction to the first row, i.e., to the subposet
\[
[n-1] \simeq \{(0,1) < (0,2) < \dots < (0,n)\} \subset J(n)
\]
induces a projection functor $h : S_n(\EuScript C) \rightarrow \text{Fun}([n-1], \EuScript C)$. Similarly, restriction to the last column
\[
[n-1] \simeq \{(0,n) < (1,n) < \dots < (n-1,n)\} \subset J(n)
\]
induces a projection functor $v : S_n(\EuScript C) \rightarrow \text{Fun}([n-1], \EuScript C)$.

\end{ex}

The following result regarding these functors is well-known.

\begin{prop}{\cite[Rem. 4.1.2]{lurie-waldhausen}}\label{prop:1}
    The functors
    \begin{gather*}
    h : S_n(\EuScript C) \rightarrow \text{Fun}([n-1], \EuScript C)\\
    v: S_n(\EuScript C) \rightarrow \text{Fun}([n-1], \EuScript C)
    \end{gather*}
    are trivial Kan fibrations. In particular, $S_n(\EuScript C)$ is a stable $\infty$-category.
    \qed
\end{prop}

\paragraph{The $\infty$-category of coherent complexes.} In the following, we recall a cubical model for coherent complexes of finite length in a stable $\infty$-category $\EuScript C$. A more general discussion of $n$-cubes in $\EuScript C$ can be found in \cite[Appendix A]{dyck_cubes}.

\begin{notat}
    Let $n \geq 0$. Consider the finite poset $I^n$, where $I:=[1]$. For $\sigma = (\sigma_1,\dots,\sigma_n) \in I^n$, we write
    \[
    \vert \sigma \vert = \sigma_1 + \dots + \sigma_n.
    \]
    For $0 \leq i \leq n$, we denote
    \[
    \sigma^{(i)} = (\underbrace{0,\dots,0}_{n-i},\underbrace{1,\dots,1}_{i}).
    \]
    A vertex $\sigma \in I^n$ is called \textit{degenerate} if it is not of the form $\sigma^{(i)}$ for some $i$.
\end{notat}

\begin{defn}
    Let $\Ce$ be a stable $\infty$-category and $n \geq 0$. We denote by
    \[
    \text{Ch}^n(\EuScript C) \subset \text{Fun}(I^n, \EuScript C)\]
    the full subcategory spanned by functors mapping degenerate vertices to zero objects.
     We call such a functor $C$ a \textit{coherent complex of length $n$} and $\text{Ch}^n(\EuScript C)$ the \textit{$\infty$-category of coherent complexes (of length $n$) in $\EuScript C$}.
\end{defn}

\begin{notat}
For $C \in \text{Ch}^n(\EuScript C)$, we abbreviate $C_\sigma := C(\sigma)$ and $C_i := C_{\sigma^{(i)}}$ and we will sometimes depict the complex $C$ as
\[
C_0 \rightarrow C_1 \rightarrow \dots \rightarrow C_n
\]
omitting the coherent system of null-homotopies. In analogy to chain complexes in an abelian category, we will refer to the maps $C_i \to C_{i+1}$ as the \textit{differentials} of $C$.
\end{notat}

\begin{defn}{\cite[Def. A.20]{dyck_cubes}} Let $\EuScript C$ be a stable $\infty$-category and $n \geq 0$.
\begin{enumerate}
    \item The \textit{total fiber} $\text{TotFib}(C)$ of a coherent complex $C \in \text{Ch}^n(\EuScript C)$ is the fiber of the canonical map
    \[
    C_0 \longrightarrow \lim_{|\sigma| >0} C_\sigma.
    \]
    
    \item Dually, the \textit{total cofiber} $\text{TotCof}(C)$ of a coherent complex $C \in \text{Ch}^n(\EuScript C)$ is the cofiber of the canonical map
    \[
    \underset{|\sigma| < n}{\text{colim }} C_\sigma \longrightarrow C_n.
    \]
\end{enumerate}
\end{defn}

In a stable $\infty$-category $\EuScript C$, the ordinary fiber and cofiber functors satisfy the identity $\text{Cof} \simeq \Sigma \circ\text{Fib}$. This has the following generalization to total fibers and cofibers.

\begin{prop}{\cite[Prop. A.31]{dyck_cubes}}\label{prop:totfib-totcof}
    Let $\EuScript C$ be a stable $\infty$-category and $C \in \text{Ch}^n(\Ce)$. Then there exists a canonical equivalence
    \[
    \text{TotCof}(C) \simeq (\Sigma^n \circ \text{TotFib})(C). \pushQED{\qed}\qedhere \popQED
    \]
\end{prop}

\begin{defn}
    A complex $C \in \text{Ch}^n(\EuScript C)$ is called \textit{exact} if $\text{TotFib}(C)$ is a zero object. We denote by 
    \[\text{Ch}_{\text{ex}}^n(\EuScript C) \subset \text{Ch}^n(\EuScript C)\] the full subcategory spanned by exact complexes.
\end{defn}

\begin{rem}\label{rem:exact-limit}
    Note that the total fiber of a complex $C \in \text{Ch}^n(\Ce)$ is a zero object if and only if the canonical map
    \[
    C_0 \longrightarrow \lim_{|\sigma| >0} C_\sigma
    \]
    is an equivalence. This is equivalent to $C$ being a limit diagram, i.e., a Cartesian cube.
\end{rem}

A direct consequence of $\ref{prop:totfib-totcof}$ is the following.
\begin{cor}\label{prop:2}
    For a complex $C \in \text{Ch}^n(\EuScript C)$, the following are equivalent:
    \begin{enumerate}
        \item[(i)] $C$ is exact.
        \item[(ii)] $\text{TotCof}(C)$ is a zero object.
    \end{enumerate}
\end{cor}

\begin{rem}\label{rem:cofcube}
    Let $C \in \text{Ch}^n(\Ce)$ be a complex of length $n\geq1$. Notice that we may view $C$ as a morphism
    \[
    C\vert_{I^{n-1}\times\{0\}} \longrightarrow C\vert_{I^{n-1}\times\{1\}}
    \]
    in the $\infty$-category $\text{Ch}^{n-1}(\Ce)$. We denote the cofiber of this morphism taken in $\text{Ch}^{n-1}(\Ce)$ by $\text{Cof}(C)$. We thus obtain a functor
    \[
    \text{Cof} : \text{Ch}^n(\Ce) \longrightarrow \text{Ch}^{n-1}(\Ce).
    \]
    Furthermore, the diagrams
    \[\begin{tikzcd}
        \text{Ch}^n(\Ce) \ar[r,"\text{Cof}"] \ar[rd,"-\vert _{I^{n-1}\times \{1\}\setminus\{\sigma^{(1)}\}}"'] & \text{Ch}^{n-1}(\Ce) \ar[d,"-\vert_{I^{n-1}\setminus\{\sigma^{(0)}\}}"]\\
        & \text{Fun}(I^{n-1}\setminus\{\sigma^{(0)}\},\Ce)
    \end{tikzcd}\]
    and
    \[
    \begin{tikzcd}
        \text{Ch}^n(\Ce) \ar[r,"\text{Cof}"] \ar[d,"-\vert_{\{0\}^{n-1}\times I}"'] & \text{Ch}^{n-1}(\Ce) \ar[d,"-\vert_{I^{n-1}\setminus\{\sigma^{(0)}\}}"]\\
        \text{Fun}(I,\Ce) \ar[r,"\text{Cof}"] & \Ce
    \end{tikzcd}
    \]
    commute.
\end{rem}

The following inductive characterization follows from \cite[Prop. 1.2.4.15]{lurie-ha}.

\begin{lem}{\cite[Prop. A.11]{dyck_cubes}}\label{lem:1}
    Let $n \geq 1$ be an integer. Then, a complex $C\in \text{Ch}^n(\Ce)$ is exact if and only if $\text{Cof}(C)$ is exact.\qed
\end{lem}

\begin{prop}\label{prop:totfib}
    Let $n \geq 1$ and $C \in \text{Ch}_{\text{ex}}^n(\Ce)$. Then, we have
    \[
    C_0 = \text{TotFib}(C\vert_{I^{n-1}\times\{1\}}).
    \]
\end{prop}
\begin{proof}
     Consider the complex $\text{Cof}(C) \in \text{Ch}^{n-1}(\EuScript C)$. By Lemma \ref{lem:1}, $\text{Cof}(C)$ is exact and we have \[
     \text{Cof}(C)\vert_{I^{n-1} \setminus\{\sigma^{(0)}\}} = C\vert_{I^{n-1}\times\{1\} \setminus \{\sigma^{(1)}\}}.\]
    By Remark \ref{rem:exact-limit}, we have
    \[
    \text{Cof}(C)_0 = \lim_{\vert\sigma\vert > 0} \text{Cof}(C)_\sigma = \lim_{\vert\sigma\vert > 0} (C\vert_{I^{n-1}\times\{1\}})_\sigma
    \]
    and the canonical map
    \[
    C_1 \longrightarrow \lim_{\vert\sigma\vert > 0} (C\vert_{I^{n-1}\times\{1\}})_\sigma
    \]
    is given by the canonical map
    \[
    C_1 \longrightarrow \text{Cof}\{C_0 \rightarrow C_1\} = \text{Cof}(C)_0.
    \]
    Hence, the fibers of the maps above coincide, which proves the claim.
\end{proof}

\begin{prop}\label{prop:4}
    For $n \geq 1$, the functor \[\text{Ch}^n_{\text{ex}}(\EuScript C) \longrightarrow \text{Ch}^{n-1}(\EuScript C)\] induced by restriction to $I^{n-1}\times\{1\}$ is a trivial Kan fibration.
\end{prop}
\begin{proof}
    We denote by $\text{Fun}'(I^n, \EuScript C) \subset \text{Fun}(I^n,\EuScript C)$ the full subcategory spanned by Cartesian cubes $Y$ such that $Y(\sigma) = 0$ for $\sigma \in (I^{n-1}\times\{0\}) \setminus \{\sigma^{(0)}\}$. Applying Lemma \ref{lem:2} twice, we deduce that restriction to $I^{n-1} \times\{1\}$ induces a trivial Kan fibration
    \[
    \text{Fun}'(I^n, \EuScript C) \longrightarrow \text{Fun}(I^{n-1}, \EuScript C).
    \]
    Furthermore, $\text{Ch}^n_\text{ex}(\EuScript C)$ is a full subcategory of $\text{Fun}'(I^n, \EuScript C)$ (by \ref{prop:2}) consisting precisely of those functors $Y \in \text{Fun}'(I^n, \EuScript C)$ such that $Y\vert_{I^{n-1}\times\{1\}} \in \text{Ch}^{n-1}(\EuScript C)$. Therefore, restriction to $I^{n-1} \times\{1\}$ induces a trivial Kan fibration
    \[
    \text{Ch}^n_\text{ex}(\EuScript C) \longrightarrow \text{Ch}^{n-1}(\EuScript C).\qedhere
    \]
\end{proof}

\paragraph{An equivalence of $\infty$-categories.}\label{par:equiv} The purpose of this subsection is a direct proof of the following proposition, which was stated in this form in \cite[Prop. 3.21]{dyck_pervscho}:

\begin{prop}
    Let $\EuScript C$ be a stable $\infty$-category and $n \geq 0$. There exists an equivalence of $\infty$-categories
    \[
    S_n(\EuScript C) \longrightarrow \text{Ch}^{n}_{\text{ex}}(\EuScript C)
    \]
    taking a Waldhausen diagram $X \in S_n(\EuScript C)$ to an exact coherent complex of the form
    \[
    X_{0,n} \rightarrow X_{n-1,n} \rightarrow X_{n-2,n-1}[1] \rightarrow \dots \rightarrow X_{01}[n-1].
    \]
\end{prop}
The case $n=0$ is trivial. For $n \geq 1$, we prove the following more refined statement.
\begin{prop}\label{thm:1}
    There exists an equivalence of $\infty$-categories
    \[
    \Theta : S_n(\EuScript C) \longrightarrow \text{Ch}^{n-1}(\EuScript C)
    \]
    such that for $X \in S_n(\Ce)$, the following hold:
    \begin{enumerate}
        \item For $0 \leq i \leq n-1$, the object $X_{i,n}$ is the total fiber of $\Theta(X)\vert_{\{0\}^i \times I^{n-i-1}}$.
        \item For $0 \leq i \leq n-1$, the object $\Theta(X)_i$ is equivalent to $X_{n-i-1, n-i}[i]$.
    \end{enumerate}
\end{prop}

\begin{notat}
For $n \geq 0$, we denote by
\[
L(n) := I^{n}\times\{n\} \cup \{\sigma \in I^n \times[n] \mid \sigma_k = 0 \text{ for }1 \leq k \leq \sigma_{n+1}\} \subset I^n\times[n]
\]
a poset with order $\sigma \leq \tau$ iff $\sigma_k \leq \tau_k$ for all $k$. As usual, we consider $L(n)$ as a category. There are injections
\begin{align*}
&\hspace{3cm}&\lambda_n : I^{n+1} &\longrightarrow L(n),& \sigma &\longmapsto (\sigma_1,\dots,\sigma_n,n\cdot \sigma_{n+1}),&&\hspace{3cm}\\
&\hspace{3cm}&i_n : L(n-1) &\longrightarrow L(n),& \sigma &\longmapsto (0,\sigma_1,\dots,\sigma_n,\sigma_{n+1} + 1). &&\hspace{3cm}
\end{align*}
We now recursively define a collection $\{\Le _n\}_{n\geq0}$ of auxiliary functor $\infty$-categories: Let $\Le_0 = \text{Fun}(L(0),\Ce) \simeq \Ce$. For $n \geq 1$, we let \[\Le_n \subset \text{Fun}(L(n),\EuScript C)\] the full subcategory spanned by diagrams $F$ such that $F \circ \lambda_n$ is an exact complex and $F\circ i_{n} \in \Le_{n-1}$.
\end{notat}

\newpage

\begin{ex}
    For $n =0,1,2,3$, we have:
    
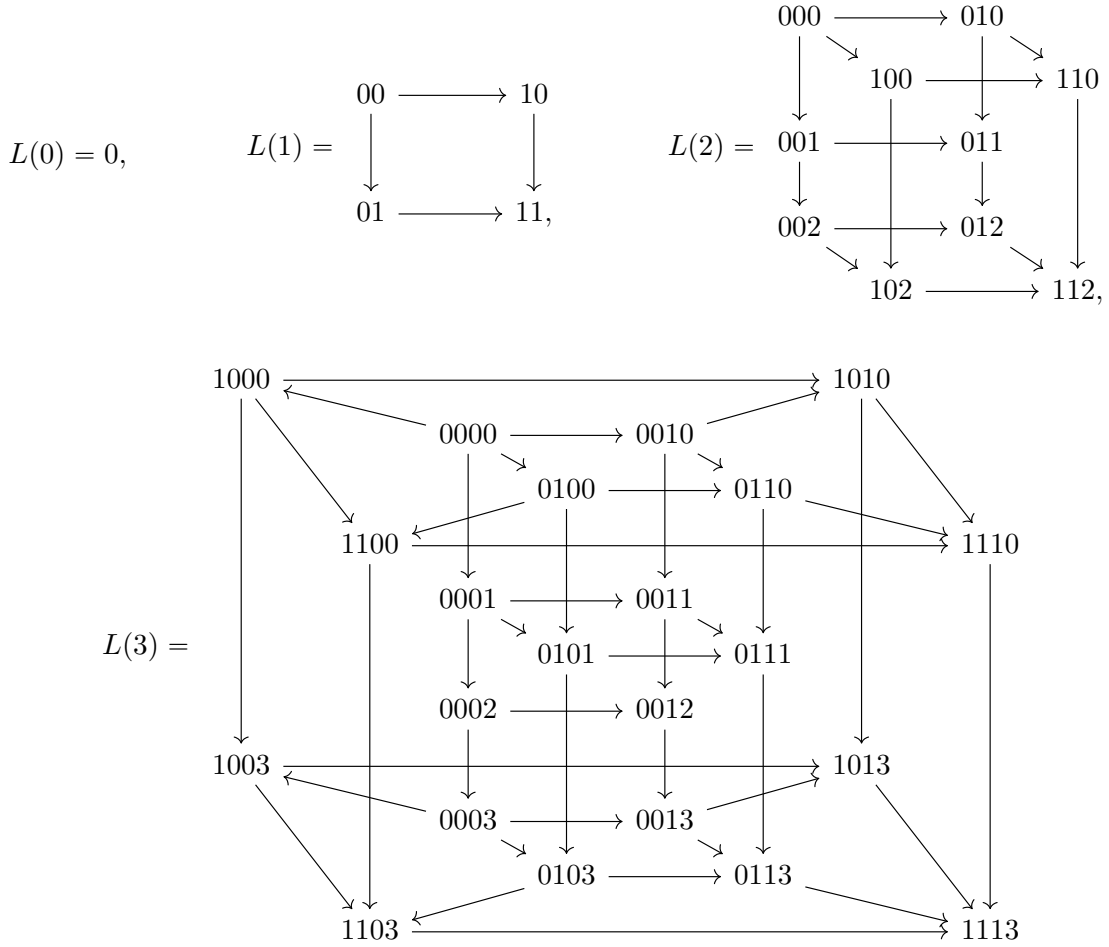
\begin{figure}[h]
        \begin{minipage}{.2\textwidth}
        \centering
            $L(0) = 0,$
        \end{minipage}
        \begin{minipage}{.3\textwidth}
        \hspace{5mm}
            $L(1) =$
            \begin{tikzcd}[row sep = large, column sep = large]
                00 \arrow[r] \arrow[d] & 10 \arrow[d]\\
                01 \arrow[r] & 11,
            \end{tikzcd}
        \end{minipage}
        \begin{minipage}{.5\textwidth}
            \centering
            $L(2) =$
            \begin{tikzcd}[row sep = 3mm, column sep = 3mm]
                000\ar[rr] \ar[rd] \ar[dd] & & 010 \ar[dd] \ar[rd] &\\
                & 100 \ar[dddd] \ar[rr] & & 110 \ar[dddd]\\
                001 \ar[rr] \ar[dd] & & 011 \ar[dd] &\\
                &&&\\
                002 \ar[rd] \ar[rr] & & 012 \ar[rd] &\\
                & 102 \ar[rr] & & 112,
            \end{tikzcd}
        \end{minipage}\\
        \vspace{0.5cm}
        
        \centering
        $L(3) = $
        \begin{tikzcd}[row sep = 2mm, column sep = 2mm]
	1000 &&&&&&&& 1010 &&& \\
	&&&& 0000 && 0010 \\
	&&&&& 0100 && 0110 \\
	&&& 1100 &&&&&&&& 1110 \\
	&&&& 0001 && 0011 \\
	&&&&& 0101 && 0111 \\
	&&&& 0002 && 0012 \\
	1003 &&&&&&&& 1013 \\
	&&&& 0003 && 0013 \\
	&&&&& 0103 && 0113 \\
	&&& 1103 &&&&&&&& 1113
	\arrow[from=1-1, to=1-9]
	\arrow[from=1-1, to=4-4]
	\arrow[from=1-1, to=8-1]
	\arrow[from=1-9, to=4-12]
	\arrow[from=1-9, to=8-9]
	\arrow[from=2-5, to=1-1]
	\arrow[from=2-5, to=2-7]
	\arrow[from=2-5, to=3-6]
	\arrow[from=2-5, to=5-5]
	\arrow[from=2-7, to=1-9]
	\arrow[from=2-7, to=3-8]
	\arrow[from=2-7, to=5-7]
	\arrow[from=3-6, to=3-8]
	\arrow[from=3-6, to=4-4]
	\arrow[from=3-6, to=6-6]
	\arrow[from=3-8, to=4-12]
	\arrow[from=3-8, to=6-8]
	\arrow[from=4-4, to=4-12]
	\arrow[from=4-4, to=11-4]
	\arrow[from=4-12, to=11-12]
	\arrow[from=5-5, to=5-7]
	\arrow[from=5-5, to=6-6]
	\arrow[from=5-5, to=7-5]
	\arrow[from=5-7, to=6-8]
	\arrow[from=5-7, to=7-7]
	\arrow[from=6-6, to=6-8]
	\arrow[from=6-6, to=10-6]
	\arrow[from=6-8, to=10-8]
	\arrow[from=7-5, to=7-7]
	\arrow[from=7-5, to=9-5]
	\arrow[from=7-7, to=9-7]
	\arrow[from=8-1, to=8-9]
	\arrow[from=8-1, to=11-4]
	\arrow[from=8-9, to=11-12]
	\arrow[from=9-5, to=8-1]
	\arrow[from=9-5, to=9-7]
	\arrow[from=9-5, to=10-6]
	\arrow[from=9-7, to=8-9]
	\arrow[from=9-7, to=10-8]
	\arrow[from=10-6, to=10-8]
	\arrow[from=10-6, to=11-4]
	\arrow[from=10-8, to=11-12]
	\arrow[from=11-4, to=11-12]
\end{tikzcd}
    \end{figure}\\
    Note that $I^{n+1}$ is embedded via $\lambda_n$ into $L(n)$ as the outer cube, while $L(n-1)$ is embedded via $i_n$ into the back face of this cube.
\end{ex}

\begin{lem}\label{prop:8}
    The functor
    \[
    \EuScript L_n \longrightarrow \text{Ch}^n(\EuScript C)
    \]
    induced by restriction to $I^n\times\{n\}$ is a trivial Kan fibration.
\end{lem}
\begin{proof}
    We prove this by induction. For $n = 0$, the claim is trivial. For $n\geq 1$, define $H(n) = I^n\times\{n\} \cup \text{im}(i_n)$ and let $\He_n$ be the image of the restriction functor
    \[
    \Le_n \longrightarrow \text{Fun}(H(n),\Ce).
    \] Notice that we have a diagram
    \[
    \begin{tikzcd}
        \Le_n \ar[r] \ar[d] & \text{Ch}^{n+1}_\text{ex}(\Ce) \ar[d, "-\vert_{I^n\times \{1\}}"]\\
        \He_n \ar[r] \ar[d] & \text{Ch}^n(\Ce) \ar[d, "-\vert_{\{0\}\times I^{n-1}}"]\\
        \Le_{n-1} \ar[r] & \text{Ch}^{n-1}(\Ce)
    \end{tikzcd}
    \]
    were both squares are pullback squares of simplicial sets. Since the map
    \[
    \Le_{n-1} \longrightarrow \text{Ch}^{n-1}(\Ce)
    \]
    is a trivial Kan fibration by the induction hypothesis, and the map
    \[
    \text{Ch}^{n+1}_\text{ex}(\Ce) \longrightarrow \text{Ch}^n(\Ce)
    \]
    is a trivial Kan fibration by Proposition \ref{prop:4}, it follows that the diagonal of the upper square, given by the restriction functor
    \[
    \Le_n \longrightarrow \text{Ch}^{n}(\Ce)
    \]
    is a trivial Kan fibration. \qedhere
\end{proof}

\begin{lem}\label{prop:9}
    The functor
    \[
    \Le_n \longrightarrow \text{Fun}([n], \EuScript C)
    \]
    induced by precomposition with the injection
    \[
    l_n : [n] \longrightarrow L(n), i \longmapsto (0,\dots,0,i)
    \]
    is a trivial Kan fibration.
\end{lem}
\begin{proof}
    As previously, we prove this by induction. The claim for $n=0$ is trivial. For $n \geq 1$, we define $K(n) = \{0\}\times I^{n-1} \times \{0,n\} \cup \text{im}(i_{n})$ and $M(n) = \text{im}(l_{n})\cup \text{im}(i_{n})$. We let $\Ke_n$ and $\Me_n$ denote the images of the functors
    \[
        \Le_n \longrightarrow \text{Fun}(K(n),\Ce)\\
    \]
    and
    \[
        \Le_n \longrightarrow \text{Fun}(M(n),\Ce)
    \]
    respectively. Then, there is a diagram
    \[
    \begin{tikzcd}
        \Le_n \ar[r] \ar[d] & \text{Ch}^{n+1}_\text{ex}(\Ce) \ar[d, "{-\vert_{\{0\} \times I^{n}}}"]\\
        \Ke_n \ar[r] \ar[d] & \text{Ch}^n(\Ce)\\
        \Me_n \ar[r] \ar[d] & \Le_{n-1} \ar[d]\\
        \text{Fun}([n],\Ce) \ar[r] & \text{Fun}([n-1],\Ce)
    \end{tikzcd}
    \]
    where the upper and lower square are pullback squares of simplicial sets and the left column composes to the functor induced by precomposition with $l_n$. Applying the dual of Proposition \ref{prop:4} and the induction hypothesis, it follows that the maps
    \[
    \Le_n \longrightarrow \Ke_n
    \]
    and \[
    \Me_n \longrightarrow \text{Fun}([n],\Ce)
    \]
    are trivial Kan fibrations. It remains to show that the map
    \[
    \Ke_n \longrightarrow \Me_n
    \]
    is a trivial Kan fibration. This is achieved by extending by zeros, applying Lemma \ref{lem:2} twice. We leave the details to the reader. \qedhere
\end{proof}

\begin{proof}[Proof of Theorem \ref{thm:1}]
    Define $\Theta$ as the composition
    \[
    S_n(\EuScript C) \overset{v}\longrightarrow \text{Fun}([n-1],\EuScript C) \overset{s}\longrightarrow \Le_{n-1} \overset{r}\longrightarrow \text{Ch}^{n-1}(\EuScript C)
    \]
    where $s$ is a section of the trivial Kan fibration
    \[
    \Le_{n-1} \longrightarrow \text{Fun}([n-1],\Ce)
    \]
    and $r$ is the functor induced by restriction to $I^{n-1} \times \{n-1\}$. As a composite of equivalences, $\Theta$ defines an equivalence of $\infty$-categories.\\

    We prove (1) and (2) by induction. For $n = 1$, we have
    \[
    \Theta(X)_0 = F(0) = F(l_0(0)) = v(X)(0) = X_{0,0},
    \]
    showing both (1) and (2).\\
    Let $n \geq 2$ and $X \in S_n(\EuScript C)$. Consider the diagram $F:=s(v(X)) \in \Le_{n-1}$. Applying Proposition \ref{prop:totfib} and the fact that $F \circ \lambda_{n-1}$ is exact, we deduce that $F(0,\dots,0) = X_{0,n}$ is the total fiber of the complex $r(F) = \Theta(X)$, thus proving (1) for $i= 0$.
    Note that we have a diagram of $\infty$-categories
    \[
    \begin{tikzcd}
        S_n(\Ce) \ar[r] \ar[d] & \text{Fun}([n-1],\Ce) \ar[r] \ar[d] & \Le_{n-1} \ar[r] \ar[d] & \text{Ch}^{n-1}(\Ce) \ar[d]\\
        S_{n-1}(\Ce) \ar[r] & \text{Fun}([n-2],\Ce) \ar[r] & \Le_{n-2} \ar[r] & \text{Ch}^{n-2}(\Ce)
    \end{tikzcd}
    \]
    where the map
    \[
    S_n(\Ce) \longrightarrow S_{n-1}(\Ce)
    \]
    is induced by the inclusion
    \[
    j_n : J(n-1) \longrightarrow J(n), \quad (i,j) \longmapsto (i+1,j+1).
    \]
    (In fact, the outer squares are strictly commutative while the middle square may not be.) The rows compose to the map $\Theta$ in the respective dimensions and we abuse notation by denoting them with the same symbol.
    For $1 \leq i \leq n-1$, using the induction hypothesis, we have
    \[
    \text{TotFib}(\Theta(X)\vert_{\{0\}^i\times I^{n-i-1}}) = \text{TotFib}(\Theta(X\circ j_n)\vert_{\{0\}^{i-1}\times I^{n-i-1} }) = (X\circ j_n)_{i-1,n-1} = X_{i,n},
    \]
    which proves (1).
    Moreover, for $0 \leq i \leq n-2$, we have
    \[
    \Theta(X)_i = (\Theta(X)\vert_{\{0\}\times I^{n-2}})_i \simeq (X\circ j_n)_{n-i-2,n-i-1}[i] = X_{n-i-1,n-i}[i].
    \]
    It remains to show condition (2) for $i = n-1$. Since $F \circ \lambda_{n-1}$ is exact, we have
    \begin{align}
    \Theta(X)_{n-1} &= \text{TotCof}(F\circ \lambda_{n-1} \vert_{\{0\}\times I^{n-1}}) \notag\\
    &\simeq \text{TotFib}(F\circ \lambda_{n-1} \vert_{\{0\}\times I^{n-1}})[n-1] \tag{$*$}
    \end{align}
    by Proposition \ref{prop:totfib-totcof}. Consider the map
    \[
    p : I^{n} \longrightarrow \Le_{n-1}, \quad
    \sigma \longmapsto \begin{cases}
        \lambda_{n-1}(0,\sigma_1,\dots,\sigma_{n-1}) &\text{for }\sigma_n = 0\\
        i_{n-1}\circ \lambda_{n-2}(\sigma_1,\dots,\sigma_{n-1}) &\text{else.}
    \end{cases}
    \]
    Then, $F \circ p$ is an object of $\text{Fun}(I^n,\Ce)$ corresponding to a morphism
    \[
    (F\circ p)\vert_{I^{n-1}\times\{0\}} = (F\circ \lambda_{n-1})\vert_{\{0\}\times I^{n-1}} \longrightarrow (F\circ p)\vert_{I^{n-1} \times \{1\}} = F\circ i_{n-1} \circ \lambda_{n-2}.
    \]
    Its fiber $\text{Fib}(F\circ p)$ admits a canonical map
    \[
    \text{Fib}(F\circ p) \longrightarrow (F\circ \lambda_{n-1})\vert_{\{0\}\times I^{n-1}}
    \]
    corresponding to a functor $G \in \text{Fun}(I^n, \Ce)$. In fact, $G \in \text{Ch}^n(\Ce)$. Note that $\text{Cof}(G) = F \circ i_{n-1} \circ \lambda_{n-2}$ is exact by definition of $\Le_{n-1}$, so that by $\ref{lem:1}$, $G$ is exact. Applying Proposition \ref{prop:totfib}, there is an equivalence
    \begin{align*}
    \text{TotFib}((F\circ \lambda_{n-1})\vert_{\{0\}\times I^{n-1}}) &= \text{TotFib}(G\vert_{I^{n-1}\times\{1\}})\\
    &\simeq G_0\\
    &= \text{Fib}\{X_{0,n} \to X_{1,n}\}\\
    &= X_{0,1}.
    \end{align*}
    Combining this with the equivalence $(*)$, we obtain the desired result. \qedhere
\end{proof}

\newpage

\subsection{Semiorthogonal decompositions of stable $\infty$-categories}\label{subsec:sod}
In the following, we present the notion of semiorthogonal decompositions of length $n$. This is an extension of the framework of semiorthogonal decompositions described in \cite{dyck_spherical}, and the structure of our discussion will closely follow the one found there.

\paragraph{Orthogonal tuples of subcategories and the span of a subcategory.} Let $\EuScript C$ be a stable $\infty$-category and $\EuScript A \subset \EuScript C$ a subcategory. We denote by $\EuScript A^\text{sf} \subset \EuScript C$ the full subcategory spanned by objects $c \in \EuScript C$ such that there is an object $a \in \EuScript A$ and an equivalence $c \simeq a$. The subcategory $\EuScript A^\text{sf}$ is called the \textit{strictly full closure} of $\EuScript A$. If $\EuScript A = \EuScript A^\text{sf}$, we say that $\EuScript A$ is \textit{strictly full}. A full subcategory $\EuScript A \subset \EuScript C$ will be called \textit{stable} if it has a zero object and is closed under fibers and cofibers. In particular, $\EuScript A$ then defines a stable $\infty$-category.\\

We say that stable subcategories $\EuScript A \subset \EuScript C$ and $\EuScript B \subset \EuScript C$ are \textit{orthogonal}, if for all objects $a \in \EuScript A, b \in \EuScript B$, the space $\text{Map}_{\EuScript C}(b, a)$ is contractible.
In particular, $\Ae$ and $\Be$ being orthogonal is equivalent to the assertion that $\text{h}\EuScript A$ and $\text{h}\EuScript B$ are orthogonal as full triangulated subcategories of $\text{h}\EuScript C$, i.e., we have $\text{Hom}_{\text{h}\EuScript C}(b,a) = 0$ for all $a \in \text{h}\EuScript A, b \in \text{h}\EuScript B$. Furthermore, if $\EuScript A$ and $\EuScript B$ are orthogonal, then the intersection $\EuScript A \cap \EuScript B$ consists only of zero objects. We say that a tuple $(\EuScript A_0, \dots, \EuScript A_n)$ of stable subcategories of $\EuScript C$ is \textit{orthogonal} if for all $i < j$, $\EuScript A_i$ and $\EuScript A_j$ are orthogonal.

Given a stable subcategory $\EuScript A \subset \EuScript C$, we define its \textit{left} and \textit{right orthogonals} to be the full subcategories $^\perp\EuScript A, \EuScript A^\perp \subset \EuScript C$ whose objects are specified by:
\begin{gather*}
    ^\perp\EuScript A = \{c \in \EuScript C \mid \text{Map}_\EuScript C(c,a) \text{ is contractible, }\:\:\forall \:a \in A \},\\
    \EuScript A^\perp = \{c \in \EuScript C \mid \text{Map}_\EuScript C(a,c) \text{ is contractible, }\:\:\forall \:a \in A \}.
\end{gather*}

For a full subcategory $\EuScript A \subset \EuScript C$, we define its \textit{stable closure} to be the smallest strictly full stable subcategory $\overline \Ae \subset \EuScript C$ containing $\EuScript A$. Then, $\text{h}\overline{\EuScript A}$ is the smallest strictly full triangulated subcategory of $\text{h}\EuScript C$ containing $\text{h}\EuScript A$, denoted by $\overline{\text{h}\Ae}$.

Given stable subcategories $\EuScript A_0, \dots , \EuScript A_n$ of $\EuScript C$, we define their \textit{sum} via
\[
\EuScript A_0 + \dots + \EuScript A_n := \overline{\EuScript A_0 \cup \dots \cup \EuScript A_n}
\]
and analogously for triangulated subcategories of a triangulated category.

\paragraph{The $\infty$-categories $\langle \EuScript A_0, \dots, \EuScript A_n \rangle$ and $\{\EuScript A_0, \dots, \EuScript A_n\}$.} Let $\EuScript C$ be a stable $\infty$-category and let $\EuScript A_0, \dots , \EuScript A_n$ be stable subcategories of $\EuScript C$. We denote by
\[
\langle \EuScript A_0, \dots, \EuScript A_n \rangle \subset S_{n+1}(\EuScript C)
\]
the full subcategory spanned by Waldhausen diagrams $X$ such that $a_i :=X_{n-i,n-i+1} \in \EuScript A_i$. Hence, an object of this $\infty$-category is a diagram of the form:

\begin{equation}
\begin{tikzcd}[row sep=large]
    a_n \arrow[r] \arrow[d] &[10pt] X_{02} \arrow[r] \arrow[d] &[6pt] X_{03} \arrow[r] \arrow[d] &[-5pt] \cdots \arrow[r] \arrow[d] &[-12pt] X_{0,n+1} \arrow[d]\\
    0 \arrow[r] & a_{n-1} \arrow[r] \arrow[d] & X_{13} \arrow[r] \arrow[d] & \ddots \arrow[r] \arrow[d] & \vdots \arrow[d] \\[1pt]
    & 0 \arrow[r] & \ddots \arrow[r]\arrow[d] & X_{n-2,n} \arrow[r] \arrow[d] & X_{n-2,n+1} \arrow[d]\\[9pt]
    & & 0 \arrow[r] & a_1 \arrow[r] \arrow[d] &X_{n-1,n+1} \arrow[d]\\[9pt]
    &&& 0 \arrow[r] & a_0
\end{tikzcd}\label{eq:2}
\end{equation}\\

\noindent
We have the $\infty$-functor $\xi : \langle \EuScript A_0, \dots, \EuScript A_n \rangle \rightarrow \EuScript C$ given by the projection to the vertex $X_{0,n+1}$. Note that, since fibers and cofibers can be computed pointwise, $\langle \EuScript A_0, \dots, \EuScript A_n \rangle$ is stable and $\xi$ is an exact functor. Further, we denote by
\[
\{\EuScript A_0, \dots, \EuScript A_n\} \subset \text{Ch}^{n}(\EuScript C)
\]
the full subcategory spanned by coherent complexes $C$ such that $a_i := C_i \in \EuScript A_i$. The total fiber restricts to an $\infty$-functor 
\[\text{TotFib} : \{\EuScript A_0, \dots, \EuScript A_n\} \longrightarrow \EuScript C.\]
As above, $\{\EuScript A_0, \dots, \EuScript A_n\}$ is stable and $\text{TotFib}$ is exact since it is a limit functor.

\begin{cor}\label{prop:3}
    \begin{itemize}
        \item[(a)] The $\infty$-categories $\langle \EuScript A_0, \dots, \EuScript A_n\rangle$ and $\{\EuScript A_0, \dots, \EuScript A_n\}$ are equivalent.
        \item[(b)] The equivalence can be chosen as to send the $\infty$-functor $\xi$ to the $\infty$-functor $\text{TotFib}$.
    \end{itemize}
\end{cor}
\begin{proof}
    Using the fact that $a_i \in \EuScript A_i$ holds if and only if $a_i[i] \in \EuScript A_i$, (a) and (b) follow directly from Proposition \ref{thm:1}.
\end{proof}

\begin{prop}\label{prop:5} Let $\EuScript C$ be a stable $\infty$-category and $\EuScript A_0, \dots, \EuScript A_n$ full stable subcategories of $\EuScript C$.
    \begin{itemize}
        \item[(a)] The $\infty$-functor in \ref{prop:3} (b) is fully faithful if and only if $(\EuScript A_0, \dots, \EuScript A_n)$ is orthogonal.
        \item[(b)] The $\infty$-functor in \ref{prop:3} (b) is essentially surjective if and only if $\EuScript A_0+\dots+ \EuScript A_n = \EuScript C$.
    \end{itemize}
\end{prop}

\begin{notat}
    Suppose we are in the situation of Proposition \ref{prop:5}. For $0 \leq i \leq n$ and an object $x \in \Ce$, we denote by $D_x^i : J(n+1) \longrightarrow \Ce$ the diagram
    
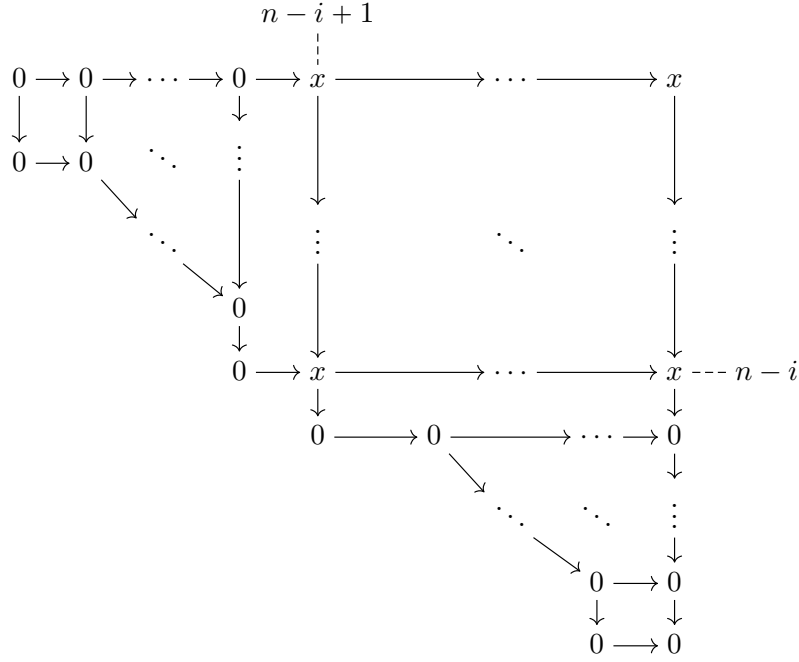
\[\begin{tikzcd}[row sep = small, column sep = small, cramped]
	&&&&[-5mm] {n-i+1} \\
	0 & 0 & \cdots & 0 & x && \cdots && x \\
	0 & 0 & \ddots & \vdots \\
	&& \ddots && \vdots && \ddots && \vdots \\
	&&& 0 \\
	&&& 0 & x && \cdots && x & {n-i} \\
	&&&& 0 & 0 && \cdots & 0 \\
	&&&&&& \ddots & \ddots & \vdots \\
	&&&&&&& 0 & 0 \\
	&&&&&&& 0 & 0
	\arrow[dashed, no head, from=1-5, to=2-5]
	\arrow[from=2-1, to=2-2]
	\arrow[from=2-1, to=3-1]
	\arrow[from=2-2, to=2-3]
	\arrow[from=2-2, to=3-2]
	\arrow[from=2-3, to=2-4]
	\arrow[from=2-4, to=2-5]
	\arrow[from=2-4, to=3-4]
	\arrow[from=2-5, to=2-7]
	\arrow[from=2-5, to=4-5]
	\arrow[from=2-7, to=2-9]
	\arrow[from=2-9, to=4-9]
	\arrow[from=3-1, to=3-2]
	\arrow[from=3-2, to=4-3]
	\arrow[from=3-4, to=5-4]
	\arrow[from=4-3, to=5-4]
	\arrow[from=4-5, to=6-5]
	\arrow[from=4-9, to=6-9]
	\arrow[from=5-4, to=6-4]
	\arrow[from=6-4, to=6-5]
	\arrow[from=6-5, to=6-7]
	\arrow[from=6-5, to=7-5]
	\arrow[from=6-7, to=6-9]
	\arrow[from=6-9, to=7-9]
	\arrow[dashed, no head, from=6-10, to=6-9]
	\arrow[from=7-5, to=7-6]
	\arrow[from=7-6, to=7-8]
	\arrow[from=7-6, to=8-7]
	\arrow[from=7-8, to=7-9]
	\arrow[from=7-9, to=8-9]
	\arrow[from=8-7, to=9-8]
	\arrow[from=8-9, to=9-9]
	\arrow[from=9-8, to=9-9]
	\arrow[from=9-8, to=10-8]
	\arrow[from=9-9, to=10-9]
	\arrow[from=10-8, to=10-9]
\end{tikzcd}\]
given explicitly by
\[
D_x^i(k,l) = \begin{cases}
    x &\text{for }k \leq n-i,\ l\geq n-i+1\\
    0 &\text{else.}
\end{cases}
\]
Note that $D_x^i \in S_{n+1}(\Ce)$, so that $D_x^i \in \langle \Ae_0, \dots ,\Ae_n\rangle$ holds iff $x \in \Ae_i$. Letting $x$ vary, one can show that this construction gives a functor
\[
D^i : \Ae_i \longrightarrow \langle \Ae_0,\dots, \Ae_n\rangle
\]
satisfying $\xi \circ D^i = i_{\Ae_i}$.
\end{notat}

\begin{proof}[Proof of Proposition]
(a): We show the implication ``$\Rightarrow$". Let $i < j$ and $a_i \in \EuScript A_i, a_j \in \EuScript A_j$. Further, let $\nu : \langle \Ae_0,\dots, \Ae_n\rangle \longrightarrow \text{Fun}(\Delta^1, \Ce)$ denote the restriction to $\{(0,n-i) < (0,n+1)\}$. Since $\xi$ factors over $\nu$, we have a commutative diagram
\[
\begin{tikzcd}
    \text{Map}_{\langle \Ae_0, \dots ,\Ae_n\rangle}(D_{a_j}^j, D^i_{a_i}) \ar[d, "\nu_*"'] \ar[rd, "\xi_*"] & \\
    \text{Map}_{\text{Fun}(\Delta^1,\Ce)}(a_j \overset{\text{id}}\to a_j,0 \to a_i) \ar[r] & \text{Map}_\Ce(a_j, a_i)
\end{tikzcd}
\]
Since $\xi_*$ is a homotopy equivalence and the Kan complex
\[\text{Map}_{\text{Fun}(\Delta^1,\Ce)}(a_j \overset{\text{id}}\to a_j,0 \to a_i)\]
is contractible, it follows that $\text{Map}_\Ce(a_j, a_i)$ is contractible.\\

We show the implication ``$\Leftarrow$". For $n=0$ this is trivially true. Let $n \geq 1$ and $C,C' \in \{\EuScript A_0, \dots ,\EuScript A_n\}$ depicted as
\begin{gather*}
C: a_0\rightarrow a_1 \rightarrow \dots \rightarrow a_n,\\
C': a'_0\rightarrow a'_1 \rightarrow \dots \rightarrow a'_n.
\end{gather*}
We have the canonical maps
\begin{gather*}
    f : a_0 \longrightarrow \lim_{\vert\sigma\vert > 0} C_\sigma =: b,\\
    f' : a'_0 \longrightarrow \lim_{\vert\sigma\vert > 0} C'_\sigma =: b'.
\end{gather*}
Using the equivalence $\text{Fib}(f) \simeq \text{Cof}(f[-1])$,
we have homotopy equivalences
\begin{align*}
    \text{Map}_\EuScript C(\text{TotFib}(C), \text{TotFib}(C')) &= \text{Map}_\EuScript C(\text{Fib}(f), \text{Fib}(f'))\\
    &\simeq \text{Map}_\EuScript C(\text{Cof}(f[-1]), \text{Fib}(f'))\\
    &\simeq \text{Fib}^T\{\text{Map}_\EuScript C(b[-1], \text{Fib}(f')) \rightarrow \text{Map}_\EuScript C(a_0[-1], \text{Fib}(f'))\}\\
    &\simeq \text{Fib}^T\left\{\begin{tikzcd}
	{\text{Fib}^T\{\text{Map}_\EuScript C(b[-1], a_0') \rightarrow \text{Map}_\EuScript C(b[-1], b')\}} \ar[d]\\
	{\text{Fib}^T\{\text{Map}_\EuScript C(a_0[-1], a_0') \rightarrow \text{Map}_\EuScript C(a_0[-1], b')\}}
\end{tikzcd} \right\}\\
    &\simeq  \text{Fib}^T\left\{\begin{tikzcd}
	{\text{Fib}^T\{\text{Map}_\EuScript C(b, a_0') \rightarrow \text{Map}_\EuScript C(b, b')\}} \ar[d]\\
	{\text{Fib}^T\{\text{Map}_\EuScript C(a_0, a_0') \rightarrow \text{Map}_\EuScript C(a_0, b')\}}
\end{tikzcd}\right\},
\end{align*}
where $\text{Fib}^T$ denotes the homotopy fiber in the category of spaces. By assumption $\text{Map}_\EuScript C(b, a_0')$ is contractible, so that the last iterated homotopy fiber above can be identified with the homotopy fiber product
\[
\text{Map}_\EuScript C(b,b') \times^h_{\text{Map}_\EuScript C(a_0,b')} \text{Map}_\EuScript C(a_0,a_0') \simeq \text{Map}_{\text{Fun}(\Delta^1,\Ce)}(f,f').
\]
Denoting by $D,D' \in \text{Ch}^n(\Ce)$ the respective limit diagrams associated to $b,b'$, we compute
\[
\text{Map}_{\text{Fun}(\Delta^1,\Ce)}(f,f') \simeq \text{Map}_{\text{Fun}(\Delta^1,\Ce)}(f,f') \times_{\text{Map}_\Ce(b,b')} \text{Map}_{\text{Ch}^n(\Ce)}(D,D') \simeq \text{Map}_{\text{Ch}^n(\Ce)}(C,C').
\]
This chain of identifications is compatible with the one induced by the functor $\text{TotFib}$, thus proving the claim.\\

(b): We show the implication ``$\Rightarrow$". Suppose that the functor
\[
\text{TotFib} : \{\EuScript A_0,\dots,\EuScript A_n\} \longrightarrow \EuScript C
\]
is essentially surjective. The total fiber of a coherent complex $C \in \{\EuScript A_0, \dots, \EuScript A_n\}$ can be computed as the limit of a finite diagram with values in the union $\EuScript A_0 \cup \dots \cup \EuScript A_n$. Since stable $\infty$-categories are closed under such limits, it follows that
\[
\EuScript C \subset \EuScript A_0 + \dots + \EuScript A_n
\]
and the other inclusion holds by definition.\\

For the implication ``$\Leftarrow$", suppose that
\[
\EuScript C = \EuScript A_0 + \dots + \EuScript A_n.
\]
We show that the functor
\[
\xi : \langle \EuScript A_0, \dots ,\EuScript A_n \rangle \longrightarrow \EuScript C
\]
is essentially surjective. For $0\leq i \leq n$, we have
\[\Ae_i = \text{im}(\xi \circ D^i) \subset \text{im}(\xi). \]
 Therefore, the $\text{im}(\xi)$ contains $\EuScript A_0 \cup \dots \cup \EuScript A_n$ and it suffices to prove that it is closed under fibers and cofibers. However, this immediately follows from the fact that $\xi$ is exact.
\end{proof}

\paragraph{Semiorthogonal decompositions.}
    \begin{defn}
        Let $\EuScript C$ be a stable $\infty$-category. We say that a tuple $(\EuScript A_0, \dots, \EuScript A_n)$ of full stable subcategories of $\EuScript C$ forms a \textit{semiorthogonal decomposition} (an \textit{SOD}, for short) of $\EuScript C$, if the functor $\xi : \langle \EuScript A_0, \dots, \EuScript A_n \rangle \rightarrow \EuScript C$ (or equivalently, the functor $\text{Tot-Fib} : \{\EuScript A_0,\dots, \EuScript A_n\} \rightarrow \EuScript C$) is an equivalence of $\infty$-categories.
    \end{defn}

    \begin{rem}
        As a first sanity check, note that any $\infty$-category $\EuScript C$ forms an SOD of itself. Furthermore, in the case $n=1$ we recover Definition 2.2.3 of \cite{dyck_spherical}.
    \end{rem}
    
    \begin{cor}\label{prop:6}
        Let $(\EuScript A_0, \dots, \EuScript A_n)$ be a tuple of full stable subcategories of a stable $\infty$-category $\EuScript C$. Then the following are equivalent:
        \begin{enumerate}[label=(\roman*)]
            \item $(\EuScript A_0, \dots, \EuScript A_n)$ forms a semiorthogonal decomposition of $\EuScript C$.
            \item $(\EuScript A_0, \dots, \EuScript A_n)$ is orthogonal and $\EuScript A_0+ \dots +\EuScript A_n = \EuScript C$.
            \item $(\text{h}\EuScript A_0, \dots, \text{h}\EuScript A_n)$ forms a semiorthogonal decomposition of $\text{h}\EuScript C$, i.e. $(\text{h}\EuScript A_0, \dots, \text{h}\EuScript A_n)$ is orthogonal and  $\text{h}\EuScript A_0+\dots+\text{h}\EuScript A_n = \text{h}\EuScript C$.
        \end{enumerate}
    \end{cor}
    \begin{proof}
        This follows directly from \ref{prop:5} and the following two facts:
        \begin{enumerate}
            \item $
        (\EuScript A_0, \dots, \EuScript A_n)$ is orthogonal iff $(\text{h}\EuScript A_0, \dots, \text{h}\EuScript A_n)$ is orthogonal.
        \item $\text{h}\EuScript A_0+\dots+\text{h}\EuScript A_n = \text{h}(\EuScript A_0+\dots+\EuScript A_n).$ \qedhere
        \end{enumerate}
    \end{proof}

    \begin{rem}
        We make two elementary observations about semiorthogonal decompositions:
        \begin{enumerate}
            \item Length $n$ semiorthogonal decompositions of a stable $\infty$-category $\Ce$ are in one-to-one correspondence to length $n$ semiorthogonal decompositions of the triangulated category $\text{h}\Ce$.
            \item Length $n$ semiorthogonal decompositions can be constructed inductively from semiorthogonal decompositions of strictly lower length: Given an SOD $(\Ae,\Be)$ of $\Ce$ and SODs $(\Ae_0,\dots,\Ae_{i})$ of $\Ae$ and $(\Be_0,\dots,\Be_j)$ of $\Be$, we have an SOD $
            (\Ae_0,\dots,\Ae_i,\Be_0,\dots,\Be_j)$
            of $\Ce$.
        \end{enumerate}
    \end{rem}

    \begin{prop}\label{prop:inductive}
        Let $(\EuScript A_0,\dots, \EuScript A_n)$ be a semiorthogonal decomposition of a stable $\infty$-category $\EuScript C$. Then, given integers
        \[
        0 \leq n_1 < \dots < n_k \leq n
        \]
        we have a semiorthogonal decomposition of $\Ce$ given by
        \[
        (\sum_{i=0}^{n_1} \Ae_i, \sum_{i=n_1 +1}^{n_2} \Ae_i, \dots, \sum_{i = n_k + 1}^n \Ae_i).
        \]
    \end{prop}
    \begin{proof}
        It is immediate that
        \[
        \sum_{i=0}^{n_1} \Ae_i +  \sum_{i=n_1 +1}^{n_2} \Ae_i +  \dots + \sum_{i = n_k + 1}^n \Ae_i = \Ce.
        \]
        It remains to show orthogonality. We prove this in the following form: Let $I,J \subset [n]$ with $i < j$ for all $i\in I,j \in J$. Then, for $a \in \sum_{i \in I} \Ae_i$ and $b \in \sum_{j\in J} \Ae_j$, the space $\text{Map}_\Ce(b,a)$ is contractible.\\
        
        Suppose that
        \[
        J = \{j_0 <\dots < j_r\}, \quad K= \{k_0 <\dots <k_s\}.
        \]
        By Corollary \ref{prop:6} (ii), $a$ is the total fiber of a complex $C$ of length $r$ and $b$ is the total fiber of a complex $D$ of length $s$, satisfying $C_m \in \Ae_{j_m}$ and $D_m \in \Ae_{k_m}$. We may assume that $r = s$, otherwise we add zeros to the complex of smaller length. Then, $\text{TotFib}$ induces a homotopy equivalence
        \[
        \text{Map}_{\text{Ch}^r(\Ce)}(D_m,C_m) \simeq \text{Map}_\Ce(b,a).
        \]
        However, we have
        \[
        \text{Map}_{\text{Ch}^r(\Ce)}(D_m,C_m) = \text{Map}_{\text{Fun}(I^r,\Ce)}(D_m,C_m) \simeq \underset{\longleftarrow}{\text{holim}}(\mathscr{F}),
        \]
        where $\mathscr{F} : \text{Mor}(I^r) \longrightarrow \text{Kan},\ (\sigma,\tau) \longmapsto \text{Map}_\Ce(D_m(\sigma),C_m(\tau))$. It follows from our assumptions that each of the spaces $\text{Map}_\Ce(D_m(\sigma),C_m(\tau))$ is contractible, so that $\text{Map}_{\text{Ch}^r(\Ce)}(D_m,C_m)$ and thus also $\text{Map}_\Ce(b,a)$ is contractible.
    \end{proof}

    \begin{ex}
        Let $\EuScript A$ be a stable $\infty$-category and define a stable $\infty$-category via
        \[
        \EuScript C := \text{Fun}([n],\EuScript A).
        \]
        For $0 \leq i \leq n$, we denote by $\EuScript A_i \subset \EuScript C$
        the full subcategory spanned by functors $\alpha_i : [n] \rightarrow \EuScript A$ such that $\alpha(j)$ is a zero object for $j \neq i$. Note that the subcategories $\Ae_i$ are stable. To show that $(\EuScript A_0, \dots, \Ae_n)$ forms a semiorthogonal decomposition of $\EuScript C$, it suffices to prove:
        \begin{enumerate}
            \item We have $\Ae_0 + \dots +\Ae_n = \Ce$.
            \item For $i < j$ and $\alpha_i \in \EuScript A_i, \alpha_j \in \EuScript A_j$, the space $\text{Map}_\EuScript C(\alpha_j,\alpha_i)$ is contractible.
        \end{enumerate}
        Let $\alpha \in \Ce$. For $0 \leq i \leq n$, we denote by $\alpha_{\leq i} : [n] \longrightarrow \Ce$ a diagram with $\alpha_{\leq i}\vert_{[i]} = \alpha\vert_{[i]}$ and $\alpha_{\leq i}(j) = 0$ for $j > i$. Note that there are obvious maps $f_i : \alpha_{\leq i-1} \longrightarrow \alpha_{\leq i}$ with $\alpha_i := \text{Cof}(f_i) \in \Ae_i$. We deduce that $\alpha_{\leq i-1} \in \Ae_0 + \dots +\Ae_n$ implies $\alpha_{\leq i} \in \Ae_0 + \dots +\Ae_n$. Applying this iteratively, $\alpha \in \Ae_0 + \dots +\Ae_n$ follows from $\alpha_{\leq 0} \in \Ae_0$, which proves (1). This says that $\xi : \langle \Ae_0,\dots ,\Ae_n\rangle \longrightarrow \Ce$ is essentially surjective. Explicitly in the case $n=2$, an object of the fiber $\xi^{-1}(\alpha)$ is given by a diagram of the form:
        
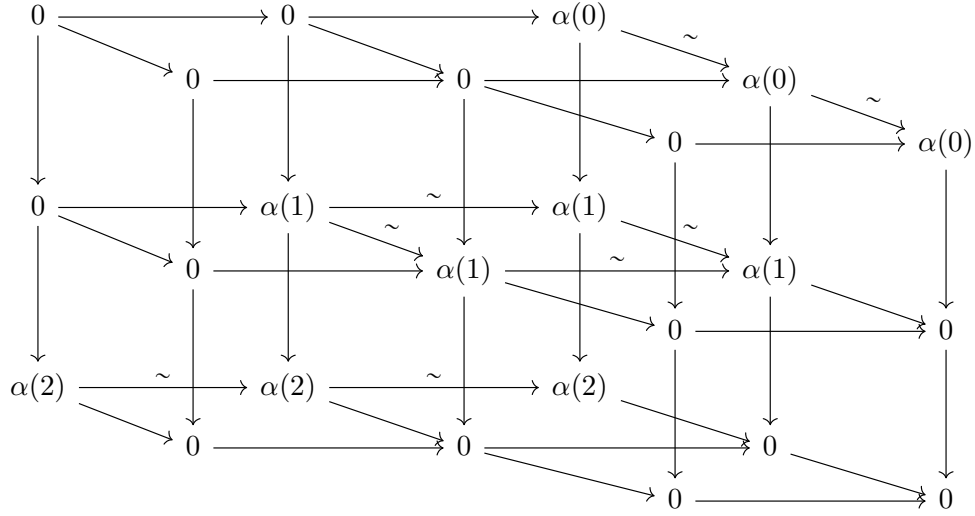
\[
        \begin{tikzcd}[row sep = tiny, column sep = small]
            0 \ar[rrr] \ar[rrd] \ar[ddd] & {} & {} & 0 \ar[rrr] \ar[rrd] \ar[ddd] & {} & {} & \alpha(0) \ar[rrd, "\sim"] \ar[ddd] & {} & {} & {} &\\
            & & 0 \ar[ddd] \ar[rrr] & & & 0 \ar[rrr] \ar[rrd] \ar[ddd] & & & \alpha(0) \ar[rrd,"\sim"] \ar[ddd] & &\\
            & & & & & & & 0 \ar[ddd] \ar[rrr] & & & \alpha(0) \ar[ddd]\\
            0 \ar[rrr] \ar[rrd] \ar[ddd] & & & \alpha(1) \ar[rrr, "\sim"] \ar[rrd,"\sim"] \ar[ddd] & & & \alpha(1) \ar[ddd] \ar[rrd, "\sim"] & & & &\\
            & & 0 \ar[rrr] \ar[ddd] & & & \alpha(1) \ar[rrr, "\sim"] \ar[rrd] \ar[ddd] & & & \alpha(1) \ar[ddd] \ar[rrd] & &\\
            & & & & & & & 0 \ar[ddd] \ar[rrr] & & & 0 \ar[ddd]\\
            \alpha(2) \ar[rrr,"\sim"] \ar[rrd] & & & \alpha(2) \ar[rrr,"\sim"] \ar[rrd] & & & \alpha(2) \ar[rrd] & & & &\\
            & & 0 \ar[rrr] & & & 0 \ar[rrr] \ar[rrd] & & & 0 \ar[rrd] & &\\
            & & & & & & & 0 \ar[rrr] & & & 0
        \end{tikzcd}\]
        
         To show (2), let $i < j$ and $\alpha_i\in \EuScript A_i, \alpha_j \in \EuScript\alpha A_j$. The space $\text{Map}_\EuScript C(\alpha_j,\alpha_i)$ is, up to homotopy equivalence, the homotopy limit of the diagram $\mathscr{F}: \text{Mor}([n]) \longrightarrow \text{Kan}, \ (k,l) \longmapsto \text{Map}_\Ce(\alpha_j(k),\alpha_i(l))$. Since $i < j$, the spaces $\mathscr{F}(k,l)$ are contractible so that $\text{Map}_\EuScript C(\alpha_j,\alpha_i)$ is contractible.
    \end{ex}

\paragraph{Locally Cartesian and coCartesian SODs.} In the following, we will study an important class of semiorthogonal decompositions, in which the morphisms "from left to right" are controllable.

\begin{notat}
Let $\EuScript C$ be a stable $\infty$-category, with stable subcategories $\EuScript A_0, \dots, \EuScript A_n \subset \Ce$. Then consider the following simplicial set $\chi(\EuScript A_0,\dots,\EuScript A_n)$: An $m$-simplex of $\chi(\EuScript A_0,\dots,\EuScript A_n)$ consists of
    \begin{enumerate}
        \item an $m$-simplex $\sigma : [m] \rightarrow [n]$ of $\Delta^{n}$,
        \item an $m$-simplex $\alpha : [m] \rightarrow \EuScript C$ such that for every $0 \leq i \leq n$, we have $\alpha\vert_{\sigma^{-1}(i)} \subset \Ae_i$.
    \end{enumerate}
We denote by $p : \chi(\EuScript A_0,\dots,\EuScript A_n) \rightarrow \Delta^{n}$ the forgetful map to part (1) of the data. By construction, $p^{-1}(i) = \EuScript A_{i}$. It follows immediately from the inner horn filling property of $\EuScript C$ that $\chi(\EuScript A_0,\dots,\EuScript A_n)$ is an $\infty$-category and this implies that $p$ is an inner fibration (see \cite[Prop. 2.3.1.5]{lurie-htt}).
\end{notat}

\begin{defn}
    A semiorthogonal decomposition $(\EuScript A_0, \dots, \EuScript A_n)$ is called \textit{locally Cartesian} (resp. \textit{locally coCartesian}) if the inner fibration $p : \chi(\EuScript A_0, \dots, \EuScript A_n) \rightarrow \Delta^n$ is a locally Cartesian (resp. locally coCartesian) fibration.
\end{defn}

\begin{rem}
    Analogously to above, one can define the stronger notion of a \textit{Cartesian} or \textit{coCartesian} semiorthogonal decomposition. As it turns out, this condition is too restrictive for many interesting examples, hence why we choose to discuss the weaker notion. However, in the case $n=1$ the definitions coincide (see Example \ref{ex:1}).
\end{rem}

For an inner fibration $p$, the property of being locally (co)Cartesian is formulated as the existence of certain locally $p$-(co)Cartesian edges. In our case, this condition has a simpler description.

\begin{defn}\label{def:ab-cart}
    Let $\EuScript C$ be a stable $\infty$-category and $\Ae,\Be \subset \Ce$ stable subcategories. An edge $e : a \rightarrow b$ in $\EuScript C$ with $a \in \EuScript A$ and $b \in \EuScript B$ is called
    \begin{enumerate}
        \item $(\EuScript A,\EuScript B)$-\textit{Cartesian} if $e$ defines a final object of the $\infty$-category $\EuScript C_{/b} \times_\EuScript C \EuScript A$.
        \item $(\EuScript A,\EuScript B)$-\textit{coCartesian} if $e$ defines an initial object of the $\infty$-category $\EuScript C_{a/} \times_\EuScript C \EuScript B$.
    \end{enumerate}
\end{defn}

\begin{prop}\label{prop:7}
    Let $\EuScript C$ be a stable $\infty$-category with stable subcategories $\Ae,\Be \subset \Ce$. Let $e : a \rightarrow b$ be an edge of $\EuScript C$ with $a \in \EuScript A$ and $b \in \EuScript B$. Then, the following are equivalent:
    \begin{enumerate}[label=(\roman*)]
        \item $e$ is $(\EuScript A,\EuScript B)$-Cartesian (resp. $(\EuScript A,\EuScript B)$-coCartesian).
        \item $e$, considered as an edge of $\chi(\EuScript A,\EuScript B)$, is $p$-Cartesian (resp. $p$-coCartesian).
    \end{enumerate}
\end{prop}
\begin{proof}
    We prove the equivalence for the "Cartesian" case. The "coCartesian" case is dual.

    The latter condition is equivalent to the map
    \[
    \chi(\EuScript A, \EuScript B)_{/e} \longrightarrow \chi(\EuScript A, \EuScript B)_{/b} \times_{\Delta^1_{/1}} \Delta^1_{/\{0\rightarrow 1\}}
    \]
    being a trivial Kan fibration. Unravelling the definition of the slice categories, this map identifies with the map
    \[
    \EuScript C_{/e} \times_\EuScript C \EuScript A \longrightarrow \EuScript C_{/b} \times_\EuScript C \EuScript A
    \]
    which is a trivial Kan fibration if and only if $e$ defines a final object in $\EuScript C_{/b} \times_\EuScript C \EuScript A$.
\end{proof}

\begin{cor}\label{cor:loccart}
Let $\EuScript C$ be a stable $\infty$-category equipped with a semiorthogonal decomposition $(\Ae_0,\dots, \Ae_n)$. Then the following are equivalent:
\begin{enumerate}[label=(\roman*)]
        \item $(\EuScript A_0,\dots,\EuScript A_n)$ is locally Cartesian (resp. locally coCartesian).
        \item For $0 \leq i < j \leq n$ and each $a_j \in \Ae_j$, there exists an $(\Ae_i,\Ae_j)$-Cartesian edge $e: a_i \to a_j$. \pushQED{\qed}\popQED
    \end{enumerate}
\end{cor}

\paragraph{Admissibility.} For a stable subcategory $\Ae \subset \Ce$ of a stable $\infty$-category, it is sensible to ask for conditions for $(\Ae,{^\perp \Ae})$ or $(\Ae^\perp , \Ae)$ to form an SOD, since ${^\perp \Ae}, \Ae^\perp$ are essentially (up to strictly full closure) the only candidates for such SODs. A solution to this is provided by the notion of admissibility. We show that the property of an SOD being locally (co)Cartesian can also be formulated in terms of admissibility conditions.

\begin{defn}\label{defn:1}
Let $\Ce$ be a stable $\infty$-category and $\Ae \subset \Ce$ a stable subcategory. Denoting by $i : \Ae \rightarrow \Ce$ the embedding functor, $\Ae$ is called
\begin{itemize}
	\item \textit{left admissible} if $i$ has a left adjoint.
	\item \textit{right admissible} if $i$ has a right adjoint.
	\item \textit{admissible} if $i$ has both a left and a right adjoint.
\end{itemize}
\end{defn}

\begin{prop}\label{prop:admiss-cocart}
Let $\Ce$ be a stable $\infty$-category and $\Ae \subset \Ce$ a stable subcategory.
\begin{enumerate}[label=(\alph*)]
	\item $\Ae \subset \Ce$ is left admissible if and only if for every $x \in \Ce$ there exists a $(\Ce,\Ae)$-coCartesian edge $e : x \rightarrow a$. Moreover, any left adjoint $^*i$ to $i: \Ae \to \Ce$ sends $x$ to the target of a $(\Ce,\Ae)$-coCartesian edge.
	\item $\Ae \subset \Ce$ is right admissible if and only if for every $x \in \Ce$ there exists an $(\Ae,\Ce)$-Cartesian edge $e : a \rightarrow x$. Moreover, any right adjoint $i^*$ to $i: \Ae \to \Ce$ sends $x$ to the source of an $(\Ae,\Ce)$-Cartesian edge.
\end{enumerate}
\end{prop}
\begin{proof}
	We show (a), the statement (b) is dual.
	
	Suppose that $\Ae \subset \Ce$ is left admissible. By definition, this means that the contravariant Grothendieck construction $\pi : \chi(i) \rightarrow \Delta^1$ associated to the inclusion $i : \Ae \rightarrow \Ce$ is a coCartesian fibration. Note that an edge $e$ of $\chi(i)$ covering the non-degenerate edge of $\Delta^1$ consists of objects $a \in \Ae, x \in \Ce$ together with an edge $x \rightarrow i(a) = a$ in $\Ce$. We identify $e$ with the corresponding edge $x \rightarrow a$. Then, $e$ is $\pi$-coCartesian if and only if the induced map
	\[
	\chi(i)_{e/} \longrightarrow \chi(i)_{x/} \times_{\Delta^1_{0/}}\Delta^1_{\{0\to 1\}/}
	\]
	is a trivial Kan fibration. Computing the slice categories, this map identifies with the map
	\[
	\Ce_{e/} \times_{\Ce} \Ae \longrightarrow \Ce_{x/} \times_{\Ce} \Ae
	\]
	which is a trivial Kan fibration if and only if $e$ defines an initial object in $\Ce_{c/} \times_{\Ce} \Ae$.
	
	The fact that any left adjoint $^*i$ to $i$ sends $x$ to the target of a $(\Ce,\Ae)$-coCartesian edge follows from Example \ref{ex:2}.
\end{proof}

The following characterizations justify our terminology.

\begin{prop}\cite[Prop. 2.3.2]{dyck_spherical}\label{prop:admiss}
	Let $\Ce$ be a stable $\infty$-category and $\Ae \subset \Ce$ a stable subcategory.
	\begin{enumerate}[label=(\alph*)]
		\item The pair $(\Ae, {^\perp\EuScript A})$ forms an SOD of $\Ce$ if and only if $\Ae \subset \Ce$ is left admissible.
		\item The pair $(\Ae^\perp,A)$ forms an SOD of $\Ce$ if and only if $\Ae \subset \Ce$ is right admissible.
	\end{enumerate}
\end{prop}
\begin{proof}
We show statement (a), as the statement (b) is dual.

	Suppose that the pair $(\Ae, {^\perp\Ae})$ forms an SOD of $\Ce$. Let $x \in \Ce$ be an object. Then there exists a biCartesian square
	\[
	\begin{tikzcd}
	b \ar[r]\ar[d] & x \ar[d, "e"]\\
	0 \ar[r] & a
	\end{tikzcd}
	\]
	such that $a \in \Ae$ and $b \in {^\perp}\Ae$. We show that $e$ is $(\Ce,\Ae)$-coCartesian. Denoting the above square by $S^{\rhd}$ and the square without its lower right corner by $S$, the square $S^{\triangleright}$ being biCartesian is equivalent to the induced map
	\[
	\Ce_{S^{\rhd}/}\longrightarrow \Ce_{S/}
	\]
	being a trivial Kan fibration. This implies that the pullback
	\begin{equation}\label{eq:triv}
	\Ce_{S^{\rhd}/} \times_\Ce \Ae \longrightarrow \Ce_{S/} \times_\Ce 		\Ae
	\end{equation}
	is a trivial Kan fibration. Furthermore, since $b \in {^\perp}\Ae$, the map
	\[
	\Ce_{\{b\to 0\}/} \times_\Ce \Ae \longrightarrow \Ce_{b/} \times_\Ce 		\Ae
	\]
	is a trivial Kan fibration. This implies that the pullback
	\[
	\Ce_{S/} \times_\Ce \Ae \cong (\Ce_{\{b\to 0\}/} \times_\Ce \Ae) \times_{\Ce_{b/} \times_\Ce \Ae} (\Ce_{x/} \times_\Ce \Ae) \longrightarrow \Ce_{x/} \times_\Ce \Ae
	\]
    is trivial.
	We choose a section $s$ of this map. Then the pullback
	\[
	(\Ce_{S^{\rhd}/} \times_\Ce \Ae) \times_{\Ce_{S/} \times_\Ce 		\Ae} (\Ce_{x/} \times_\Ce \Ae) \longrightarrow \Ce_{x/} \times_\Ce \Ae
	\]
	of \ref{eq:triv} along $s$ is a trivial Kan fibration. This identifies with the map
	\[
	\Ce_{e/}\times_\Ce \Ae \longrightarrow \Ce_{x/} \times_\Ce \Ae
	\]
	which is a trivial Kan fibration if and only if $e$ is $(\Ce,\Ae)$-coCartesian. Alternatively, one can show that a left adjoint to the embedding $i' : \Ae \rightarrow \{\Ae, {^\perp\EuScript A}\}$ (which is identified with $i$ since $(\Ae,{^\perp\EuScript A})$ forms an SOD) is given by sending an arrow $a \to b$ to the object $a$.
	
	Conversely, suppose that $\Ae \subset \Ce$ is left admissible. Since $(\Ae, {^\perp\EuScript A})$ is orthogonal, it suffices to construct for every $x \in \Ce$, a biCartesian square
	\[
	\begin{tikzcd}
	x \ar[r]\ar[d] & 0 \ar[d]\\
	a \ar[r] & b
	\end{tikzcd}
	\]
	such that $a \in \Ae$ and $b \in {^\perp\EuScript A}$. Since $\Ae \subset \Ce$ is left admissible, there exists a $(\Ce,\Ae)$-coCartesian edge $e : x \rightarrow a$. Choosing a cofiber sequence associated to $e$, we have a square as above and it remains to show that $b := \text{Cof}(e) \in {^\perp\EuScript A}$. Denoting again by $S^{\rhd}$ the cofiber sequence and by $S$ the square without its lower right corner, we analogously deduce that the map
	\[
	\Ce_{S^{\rhd}/} \times_\Ce \Ae \longrightarrow \Ce_{S/} \times_\Ce \Ae
	\]
	is a trivial Kan fibration. However, since $e$ is $(\Ce,\Ae)$-coCartesian, it is immediate that $\Ce_{S/} \times_\Ce \Ae$ is contractible and thus, $\Ce_{S^{\rhd}/} \times_\Ce \Ae$ is contractible. The map
	\[
	\Ce_{S^{\rhd}/} \times_\Ce \Ae \longrightarrow \Ce_{b/} \times_\Ce \Ae
	\]
	is also a trivial Kan fibration and therefore, $\Ce_{b/} \times_\Ce \Ae$ is contractible which implies that $b \in {^\perp\EuScript A}$.
\end{proof}

\begin{rem}
Let $\Ce$ be a stable $\infty$-category equipped with an SOD $(\Ae_0,\dots,\Ae_n)$. Combining the above proposition with \ref{prop:inductive}, it follows that $\Ae_0 \subset \Ce$ is left admissible and $\Ae_n \subset \Ce$ is right admissible. However, for $1 \leq i \leq n$, the subcategories $\Ae_i \subset \Ce$ are in general not left admissible and similarly for $0 \leq i \leq n-1$, the subcategories $\Ae_i \subset \Ce$ are in general not right admissible. However, for a locally Cartesian or locally coCartesian SOD, the situation is more controllable. In fact, we have the following characterization.
\end{rem}

\begin{prop}\label{prop:cart-admiss}
Let $\Ce$ be a stable $\infty$-category equipped with an SOD $(\Ae_0,\dots,\Ae_n)$.
\begin{enumerate}[label=(\alph*)]
	\item $(\Ae_0,\dots,\Ae_n)$ is locally coCartesian if and only if for every $1 \leq i \leq n$, $\Ae_i \subset \Ce$ is left admissible.
	\item $(\Ae_0,\dots,\Ae_n)$ is locally Cartesian if and only if for every $0 \leq i \leq n-1$, $\Ae_i \subset \Ce$ is right admissible.
\end{enumerate}
\end{prop}
\begin{proof}
    We show (a), the statement (b) is dual.
    
    The only if direction is straightforward from Proposition \ref{prop:admiss-cocart} and Corollary \ref{cor:loccart}. We reduce the if direction to the case $n=1$ via the following argument: Let $1 \leq i \leq n$. Assuming the statement for $n=1$, it would follow that $\Ae_i$ is left admissible as a subcategory $\Ae_{i-1} + \Ae_i$, since the SOD $(\Ae_{i-1},\Ae_i)$ is coCartesian. Now let $^\perp\Ae_i$ be the left orthogonal of $\Ae_i$ in $\Ae_{i-1} + \Ae_i$, so that $(\Ae_i,^\perp \Ae_i)$ forms an SOD. Using $\ref{prop:6}$, it follows that
    \[
    (\Ae_1,\dots,\Ae_{i-2},\Ae_i,^\perp \Ae_i,\Ae_{i+1},\dots,\Ae_n)
    \]
    forms an SOD of $\Ce$. Repeating the same steps with the coCartesian SODs $(\Ae_{i-2},\Ae_i), (\Ae_{i-3},\Ae_i),\dots$, we obtain an SOD
    \[
    (\Ae_i,^\perp \Ae_1,\dots,^\perp \Ae_i,\Ae_{i+1},\dots,\Ae_n)
    \]
    of $\Ce$, implying that $\Ae_i \subset \Ce$ is left admissible.\\

    To prove the case $n=1$, let $(\Ae,\Be)$ be a coCartesian SOD of $\Ce$ and let $x \in \Ce$. We show that there exists a $(\Ce,\Be)$-coCartesian edge $e : x \rightarrow b$. Since $(\Ae,\Be)$ is an SOD, there exists a biCartesian square
	\[
	\begin{tikzcd}
	a \ar[r,"f"]\ar[d] & b \ar[d]\\
	0 \ar[r] & x
	\end{tikzcd}
	\]
	such that $a \in \Ae, b \in \Be$. Since $(\Ae,\Be)$ is coCartesian, there exists an $(\Ae,\Be)$-coCartesian edge $h : a \rightarrow b'$. Up to contractible choice, there exists a unique map $h' : b' \rightarrow b$ such that $h' \circ h \simeq g$. Defining $x' := \text{Cof}(h)$, we obtain a canonical map $j : x' \rightarrow x$ such that the diagram
	\[
	\begin{tikzcd}
	a \ar[r, "h"]\ar[d] & b' \ar[r,"h'"] \ar[d] & b \ar[d]\\
	0 \ar[r] & x' \ar[r, "j"] & x
	\end{tikzcd}
	\]
	commutes. Since the left square and the outer square are biCartesian, it follows that the right square is biCartesian. We now set $b^! := \text{Cof}(j)$ to obtain a commutative diagram
	\[
	\begin{tikzcd}
	a \ar[r, "h"]\ar[d] & b' \ar[r,"h'"] \ar[d] & b \ar[d,"g"]\\
	0 \ar[r] & x' \ar[d] \ar[r, "j"] & x \ar[d,"e"]\\
	& 0 \ar[r] & b^!
	\end{tikzcd}
	\]
	where all squares are biCartesian. We claim that $e$ is $(\Ce,\Be)$-coCartesian. We denote by $S^\rhd$ the upper left square and by $S$ the square $S^\rhd$ without its lower right corner. Further, we denote by $T^\rhd$ the lower right square and by $T$ the square $T^\rhd$ without its lower right corner. Then, the natural map
	\[
	\Ce_{S^\rhd/} \times_\Ce \Be \longrightarrow \Ce_{S/} \times_\Ce \Be
	\]
	is a trivial Kan fibration. However, $\Ce_{S/} \times_\Ce \Be$ is contractible since $h$ is $(\Ae,\Be)$-coCartesian. This implies that $\Ce_{S^\rhd/} \times_\Ce \Be$ is contractible and as the natural map
	\[
	\Ce_{S^\rhd/} \times_\Ce \Be \longrightarrow \Ce_{x'/} \times_\Ce \Be
	\]
	is a trivial Kan fibration, it follows that $\Ce_{x'/} \times_\Ce \Be$ is contractible. This is equivalent to the natural map
	\[
	\Ce_{\{x'\to 0\}/} \times_\Ce \Be \longrightarrow \Ce_{x'/} \times_\Ce \Be
	\]
	being a trivial Kan fibration. Pullback along a section of the natural map
	\[
	\Ce_{j/} \times_\Ce \Be \longrightarrow \Ce_{x/} \times_\Ce \Be
	\]
	yields that the natural map
	\[
	\Ce_{T/} \times_\Ce \Be \longrightarrow \Ce_{x/} \times_\Ce \Be
	\]
	is a trivial Kan fibration. We choose a section $s$ of this map. Then, pulling back the trivial Kan fibration
	\[
	\Ce_{T^\rhd/} \times_\Ce \Be \longrightarrow \Ce_{T/} \times_\Ce \Be
	\]
	along $s$, we obtain that the natural map
	\[
	(\Ce_{T^\rhd/} \times_\Ce \Be) \times_{\Ce_{T/} \times_\Ce \Be} (\Ce_{x/} \times_\Ce \Be) \longrightarrow \Ce_{x/} \times_\Ce \Be,
	\]
	which identifies with the natural map
	\[
	\Ce_{e/} \times_\Ce \Be \longrightarrow \Ce_{x/} \times_\Ce \Be,
	\]
	is a trivial Kan fibration. \qedhere
\end{proof}

\begin{notat}
Let $\Ce$ be a stable $\infty$-category equipped with an SOD $(\Ae_0,\dots,\Ae_n)$. To summarize our results, we introduce a simplicial set $\Pi(\Ae_0,\dots,\Ae_n)$ expanding $\chi(\Ae_0,\dots,\Ae_n)$: An $m$-simplex of $\Pi(\Ae_0,\dots,\Ae_n)$ consists of
\begin{enumerate}
	\item an $m$-simplex $\sigma : [m] \rightarrow [n+1]$ of $\Delta^{n+1}$,
	\item an $m$-simplex $\alpha : \Delta^m \rightarrow \Ce$ such that for every $0\leq i \leq n$, we have $\alpha\vert\sigma^{-1}(i+1) \subset \Ae_i$.
\end{enumerate}
We denote by $\pi : \Pi(\Ae_0,\dots,\Ae_n) \rightarrow \Delta^{n+1}$ the forgetful map to part (1) of the data. By construction $\pi^{-1}(i+1) = \Ae_i$ for $0 \leq i \leq n$ and $\pi^{-1}(0) = \Ce$. It is immediate that $\Pi(\Ae_0,\dots,\Ae_n)$ defines an $\infty$-category and thus, $\pi$ is an inner fibration. We have a canonical embedding \[\iota : \chi(\Ae_0,\dots,\Ae_n) \hookrightarrow \Pi(\Ae_0,\dots,\Ae_n)\] given by the assignment $(\sigma, \alpha) \longmapsto (\sigma',\alpha)$, where $\sigma'(k) := \sigma(k) + 1$. Moreover, we have $p = \pi \circ \iota$. There is also the forgetful map to part (2) of the data, denoted by $q : \Pi(\Ae_0,\dots,\Ae_n) \to \Ce$.
\end{notat}

The following summarizes the results of this subsection.

\begin{cor}
	Let $\Ce$ be a stable $\infty$-category equipped with an SOD $(\Ae_0,\dots,\Ae_n)$. Then $(\Ae_0,\dots,\Ae_n)$ is locally coCartesian if and only if $\pi$ is a locally coCartesian fibration. \qed
\end{cor}

\paragraph{A reconstruction theorem.}\label{par:recons}
We prove a reconstruction theorem for locally coCartesian SODs.

\begin{defn}
Let $\Ce$ be a stable $\infty$-category equipped with a locally coCartesian semiorthogonal decomposition $(\EuScript A_0,\dots,\EuScript A_n)$.
We denote by \[\Se_\text{coCart} \subset \text{Fun}(\text{sd}([n]), \Pi(\Ae_0,\dots,\Ae_n))\]
the full subcategory of functors $S$ satisfying the following conditions:
\begin{enumerate}
	\item We have $S(\varnothing) \in \pi^{-1}(0) = \Ce$ and for every subdivision $n_0n_1\dots n_i$, $S(n_0n_1\dots n_i) \in \pi^{-1}(n_i+1)=\Ae_{n_i}$.
	\item For every $0 \leq i \leq n$ and a subdivision $n_0n_1\dots n_i$, the edge $S(n_0n_1\dots n_{i-1} \to n_0n_1\dots n_i)$ is locally $\pi$-coCartesian.
\end{enumerate}
We call $\Se_\text{coCart}$ the $\infty$-category of \textit{coCartesian subdivision cubes}.
\end{defn}

\begin{rem}
    The first condition above can also be formulated as $S$ being a functor over the diagram
    \[
    \text{sd}([n]) \longrightarrow \Delta^{n+1} \overset{\pi}\longleftarrow \Pi(\Ae_0,\dots,\Ae_n)
    \]
    where the functor $\text{sd}([n]) \to \Delta^{n+1}$ sends $\varnothing$ to $0$ and a subdivision $n_0 n_1\dots n_i$ to $n_i+1$.
\end{rem}

\begin{ex}
Let $\Xe$ be a stable $\infty$-category equipped with a locally coCartesian semiorthogonal decomposition $(\EuScript A,\Be,\Ce)$. Then an object of $\Se_\text{coCart}$ corresponds to a cube of the form
\[\begin{tikzcd}
	x && b \\
	& c && {c''} \\
	a && {b'} \\
	& {c'} && {c'''}
	\arrow["{!}", from=1-1, to=1-3]
	\arrow["{!}", from=1-1, to=2-2]
	\arrow["{!}"', from=1-1, to=3-1]
	\arrow["{!}", from=1-3, to=2-4]
	\arrow[from=1-3, to=3-3]
	\arrow[from=2-2, to=2-4]
	\arrow[from=2-2, to=4-2]
	\arrow[from=2-4, to=4-4]
	\arrow["{!}"{pos=0.4}, from=3-1, to=3-3]
	\arrow["{!}"', from=3-1, to=4-2]
	\arrow["{!}", from=3-3, to=4-4]
	\arrow[from=4-2, to=4-4]
\end{tikzcd}\]
where all edges marked by "$!$" are coCartesian in the appropiate sense and we have $x \in \Xe, a \in \Ae, b,b' \in \Be, c,c',c'',c''' \in \Ce$.

\end{ex}

\begin{prop}
	The functor
	\[
	\Se_\text{coCart} \longrightarrow \Ce
	\]
	evaluating at $\varnothing \in \text{sd}([n])$ is a trivial Kan fibration.
\end{prop}
\begin{proof}
For $-1 \leq i \leq n$, we define
\[\Pi_i := \pi^{-1}(\Delta^{\{0,\dots,i+1\}}) \subset \Pi(\Ae_0,\dots,\Ae_n).\] Then, $\pi$ restricts to a locally coCartesian fibration $\Pi_i \rightarrow \Delta^{i+1}$. We define an $\infty$-category \[\Se_i \subset \text{Fun}_{\Delta^{i+1}}(\text{sd}([i]),\Pi_i)\] as the full subcategory of functors $S$ over the diagram
	\[
	\text{sd}([i]) \longrightarrow \Delta^{i+1} \overset{\pi\vert_{\Pi_i}}\longleftarrow \Pi_i
	\]
	satisfying the following condition: For every $0 \leq j \leq i$ and a subdivision $n_0n_1 \dots n_j$, the edge $S(n_0n_1 \dots n_{j-1} \to n_0n_1 \dots n_j)$ is locally $\pi\vert_{\Pi_i}$-coCartesian.
    
    Note that we allow $i = -1$, which we interpret as $\Se_{-1} = \text{Fun}_{\Delta^0}(\text{sd}(\varnothing),\pi^{-1}(0)) = \Ce$. Then, the natural inclusions
    \[
    \text{sd}(\varnothing) \subset \text{sd}([0]) \subset \dots \subset \text{sd}([n])
    \]
    given by postcomposition with the natural inclusions $[i] \hookrightarrow [i+1]$ induce a chain of $\infty$-categories
	\begin{equation*}
	\Se_\text{coCart} = \Se_{n} \longrightarrow \Se_{n-1} \longrightarrow \dots \longrightarrow \Se_{0} \longrightarrow \Se_{-1} = \Ce.
	\end{equation*}
    We wish to show that every functor in this chain is a trivial Kan fibration.\\

	Let $0 \leq i \leq n$. Applying Lemma \ref{lem:relative}, we deduce that the restriction functor
	\[
	\text{Fun}'_{\Delta^{i+1}}(\text{sd}([i]), \Pi_i) \longrightarrow \text{Fun}'_{\Delta^{i+1}}(\text{sd}([i-1]), \Pi_i)
	\]
	is a trivial Kan fibration, where
	\begin{itemize}
		\item $\text{Fun}'_{\Delta^{i+1}}(\text{sd}([i]), \Pi_i) \subset \text{Fun}_{\Delta^{i+1}}(\text{sd}([i]), \Pi_i)$ is the full subcategory of functors $S$ which are $\pi\vert_{\Pi_i}$-left Kan extended from $S\vert_{\text{sd}([i-1])}$,
		\item $\text{Fun}'_{\Delta^{i+1}}(\text{sd}([i-1]), \Pi_i) \subset \text{Fun}_{\Delta^{i+1}}(\text{sd}([i-1]), \Pi_i)$ is the full subcategory of functors $S$ such that for every $f \in \text{sd}([i])$, the induced diagram
		\[
		\text{sd}([i-1])_{/f} \longrightarrow \Pi_i
		\]
		admits a $\pi\vert_{\Pi_i}$-colimit.
	\end{itemize}
	Notice that for any $n_0 n_1 \dots n_j \in \text{sd}([i])\setminus \text{sd}([i-1])$, the slice category $\text{sd}([i-1])_{/n_0 n_1 \dots n_j}$ has the final object $n_0 n_1 \dots n_{j-1} \to n_0 n_1 \dots n_j$. Given $S \in \text{Fun}_{\Delta^{i+1}}(\text{sd}([i-1]), \Pi_i)$, it follows from cofinality that the induced diagram
    \[
    \text{sd}([i-1])_{/n_0 n_1 \dots n_j} \longrightarrow \Pi_i
    \]
    admits a $\pi\vert_{\Pi_i}$-colimit if and only if there exists a $\pi\vert_{\Pi_i}$-coCartesian edge $S(n_0 n_1 \dots n_{j-1}) \to x$. For the latter, we can choose a $\pi\vert_{\Pi_i}$-coCartesian edge $S(n_0 n_1 \dots n_{j-1}) \to x$ with $\pi(x) = i+1$, which is also $\pi\vert_{\Pi_i}$-coCartesian since the only morphism $i+1 \to x$ in $\Delta^{i+1}$ is the identity on $i+1$. We deduce that in fact
    \[
    \text{Fun}'_{\Delta^{i+1}}(\text{sd}([i-1]), \Pi_i) =\text{Fun}_{\Delta^{i+1}}(\text{sd}([i-1]), \Pi_i).
    \]
    We leave it to the reader to verify that we have a pullback diagram of simplicial sets:
    \[\begin{tikzcd}
	{\Se_i} & {\text{Fun}'_{\Delta^{i+1}}(\text{sd}([i]),\Pi_i)} \\
	{\Se_{i-1}} & {\text{Fun}_{\Delta^{i+1}}(\text{sd}([i-1]),\Pi_i)}
	\arrow[hook, from=1-1, to=1-2]
	\arrow[from=1-1, to=2-1]
	\arrow[from=1-2, to=2-2]
	\arrow[hook, from=2-1, to=2-2]
\end{tikzcd}\]
Thus, the restriction $\Se_i \to \Se_{i-1}$ is a trivial Kan fibration.
\end{proof}

We will now consider the $\infty$-category 
\[\Le := \text{oplax-lim}(p) \subset \text{Fun}_{\Delta^n}(\text{sd}([n])\setminus \varnothing,\chi(\Ae_0,\dots,\Ae_n)).\] Note that we can write
\[
\text{Fun}_{\Delta^n}(\text{sd}([n])\setminus \varnothing,\chi(\Ae_0,\dots,\Ae_n)) = \text{Fun}_{\Delta^{n+1}}(\text{sd}([n])\setminus \varnothing,\Pi(\Ae_0,\dots,\Ae_n)).
\]
Therefore, we have a well-defined restriction functor
	\[
	\Se_\text{coCart} \longrightarrow \Le.
	\]
	Our next goal is to show that this defines a trivial Kan fibration. To this end, we first need a few preliminaries.

\begin{lem}\label{lem:underlying}
	The functor $S \in \text{Fun}_{\Delta^{n+1}}(\text{sd}([n]), \Pi(\Ae_0,\dots,\Ae_n))$ is a $\pi$-limit diagram if and only if $q \circ S \in \text{Fun}(\text{sd}([n]), \Ce)$ is a limit diagram.
\end{lem}
\begin{proof}
	We set $\widetilde{S} := S\vert_{\text{sd}([n])\setminus \varnothing}$. The functor $S$ is a $\pi$-limit diagram if and only if the natural map
	\[
	\Pi(\Ae_0,\dots,\Ae_n)_{/S} \longrightarrow \Pi(\Ae_0,\dots,\Ae_n)_{/\widetilde{S}} \times_{\Delta^{n+1}_{/\pi \circ \widetilde{S}}} \Delta^{n+1}_{/\pi \circ S}
	\]
	is a trivial Kan fibration. Using the fact that $\Delta^{n+1}_{/\pi \circ S}$ consists of a single object $f : \text{sd}([n]) \to \Delta^{n+1}$ with $f(\varnothing) = 0$, we compute
    \begin{align*}
        \Pi(\Ae_0,\dots,\Ae_n)_{/\widetilde{S}} \times_{\Delta^{n+1}_{/\pi \circ \widetilde{S}}} \Delta^{n+1}_{/\pi \circ S} &\cong \Pi(\Ae_0,\dots,\Ae_n)_{/\widetilde{S}} \times_{\Delta^{n+1}} \{0\}\\
        &=\Pi(\Ae_0,\dots,\Ae_n)_{/\widetilde{S}} \times_{\Pi(\Ae_0,\dots,\Ae_n)} \pi^{-1}(0).
    \end{align*}
    It is left to the reader that we have identifications
    \[
    \Pi(\Ae_0,\dots,\Ae_n)_{/S} \cong \Ce_{/q\circ S}, \quad \Pi(\Ae_0,\dots,\Ae_n)_{/\widetilde{S}} \times_{\Pi(\Ae_0,\dots,\Ae_n)} \pi^{-1}(0) \cong \Ce_{/q\circ \widetilde{S}}.
    \]
    Thus, the map above identifies with the map
	\[
	\Ce_{/q\circ S} \longrightarrow \Ce_{/q \circ \widetilde{S}}
	\]
	which is a trivial Kan fibration if and only if $q \circ S$ is a limit diagram.
\end{proof}

\begin{lem}\label{lem:induced}
	Let $\Ce$ be a stable $\infty$-category and $\Ae \subset \Ce$ a stable subcategory which is left admissible. Suppose we are given a commutative square
	\[
	\begin{tikzcd}
		x \ar[r] \ar[d,"f"] & a \ar[d,"g"]\\
		x' \ar[r,"!"] & a'
	\end{tikzcd}
	\]
	where the lower horizontal map is $(\Ce,\Ae)$-coCartesian. Then, the upper horizontal map is $(\Ce,\Ae)$-coCartesian if and only if the induced map
	\[
	\text{Cof}(f) \longrightarrow \text{Cof}(g)
	\]
	is $(\Ce,\Ae)$-coCartesian.
\end{lem}
\begin{proof}
We choose a left adjoint $^*i$ to the embedding $i : \Ae \to \Ce$ and $(\Ce,\Ae)$-coCartesian edges $x \to {^*i}(x), x' \to {^*i}(x'), x'' \to {^*i}(x'')$ where $x'' := \text{Cof}(f)$ (which exist by Proposition \ref{prop:admiss-cocart}). We obtain an induced diagram
	\[
	\begin{tikzcd}
		x \ar[r,"!"] \ar[d,"f"] & {^*i}(x) \ar[d] \ar[r,"h"] & a \ar[d,"g"]\\
		x' \ar[d] \ar[r,"!"] & {^*i}(x') \ar[d] \ar[r,"h'"] & 			a' \ar[d]\\
	x'' \ar[r,"!"] & {^*i}(x'') \ar[r,"h''"] & a''
	\end{tikzcd}
	\]
	where $a'' := \text{Cof}(g)$ and $h'$ is an equivalence. Since $^*i$ is exact (see \ref{cor:adjunction}), all three vertical sequences of the diagram are cofiber sequences. We obtain an induced cofiber sequence
	\[
	\text{Cof}(h) \to \text{Cof}(h') = 0 \to \text{Cof}(h''). 
	\]
	It follows that $\text{Cof}(h)$ is a zero object if and only if $\text{Cof}(h'')$ is a zero object. This is equivalent to our original claim.
\end{proof}

\begin{rem}
    Let $i \geq 0$ be an integer. By virtue of Example \ref{ex:sdcube}, we can view $S \in \text{Fun}(\text{sd}([i]),\Ce)$ as a functor $[1]^{i+1} \to \Ce$. Via the construction of Remark \ref{rem:cofcube}, we can associate to this the functor $\text{Cof}(S) : [1]^i \to \Ce$ corresponding to a functor $\text{sd}([i-1]) \to \Ce$. Explicitly, this functor sends a subdivision $n_0 n_1 \dots n_{j} \in \text{sd}([i-1])$ to the cofiber of
    \[
    S(m_0m_1\dots m_j \to 0m_0m_1 \dots m_j)
    \]
    where $m_k = n_k +1$.
\end{rem}

\begin{thm}
	The restriction functor
	\[
	\Se_\text{coCart} \longrightarrow \Le
	\]
	is a trivial Kan fibration.
\end{thm}
\begin{proof}
	Our first step is to show that for every $S \in \Se_\text{coCart}$, $q \circ S$ is a limit diagram (i.e. $S$ is a $\pi$-limit diagram by Lemma \ref{lem:underlying}). Specifically, we prove by induction that for every $0 \leq i \leq n$ the cube $\text{Cof}^i(q \circ S) : \text{sd}([n-i]) \to \Ce$ satisfies the following properties:
	\begin{enumerate}
		\item We have $\text{Cof}^i(q \circ S)(\varnothing) \in \Ae_i + \dots + \Ae_n$.
		\item For every $0 \leq j \leq n-i$ and a subdivision $n_0n_1\dots n_j$, the edge \[\text{Cof}^i(q\circ S)(n_0n_1\dots n_{j-1} \to n_0n_1\dots n_j)\] is $(\Ce,\Ae_{n_j+i})$-coCartesian.
	\end{enumerate}
	If the conditions above are satisfied for $i=n$, then $\text{Cof}^n(q \circ S)$ is an equivalence, i.e. a biCartesian edge. This would imply that $q \circ S$ is a biCartesian cube, by Lemma \ref{lem:1}.
	
	For $i = 0$, the conditions hold by definition of $\Se_\text{coCart}$. For $1 \leq i \leq n$, we have
	\[
	\text{Cof}^i(q \circ S)(\varnothing) = \text{Cof}\{\text{Cof}^{i-1}(q\circ S)(\varnothing) \overset{f}\longrightarrow \text{Cof}^{i-1}(q\circ S)(0)\}.
	\]
	By the inductive hypothesis, $f$ is a $(\Ce,\Ae_{i-1})$-coCartesian edge with source in $\Ae_{i-1}+\dots+\Ae_n$ and hence, $\text{Cof}(f) \in \Ae_{i-1}+ \dots+ \Ae_n$. In the proof of \ref{prop:admiss} it is shown that $f$ being $(\Ce,\Ae_{i-1})$-coCartesian implies that $\text{Cof}(f) \in {^\perp\Ae_{i-1}}$. Thus, $\text{Cof}^i(q \circ S)(\varnothing) = \text{Cof}(f) \in \Ae_i+\dots+\Ae_n$, which proves (1). 
	
	For (2), let $n_0n_1\dots n_j$ be a subdivision. We have a commutative square
	\[\begin{tikzcd}[row sep=large]
	{\text{Cof}^{i-1}(q\circ S)(m_0m_1\dots m_{j-1})} & {\text{Cof}^{i-1}(q\circ S)(m_0m_1\dots m_j)} \\
	{\text{Cof}^{i-1}(q\circ S)(0m_0m_1\dots m_{j-1})} & {\text{Cof}^{i-1}(q\circ S)(0m_0m_1\dots m_j)}
	\arrow[from=1-1, to=1-2]
	\arrow["f"', from=1-1, to=2-1]
	\arrow["g", from=1-2, to=2-2]
	\arrow[from=2-1, to=2-2]
\end{tikzcd}\]
where $m_k = n_k + 1$ and $\text{Cof}^i(q\circ S)(n_0n_1\dots n_{j-1} \to n_0n_1\dots n_j)$ is the induced map
	\[
	e : \text{Cof}(f) \longrightarrow \text{Cof}(g).
	\]
	Furthermore, by the inductive hypothesis, the horizontal maps of the square above are $(\Ce,\Ae_{m_j + i-1})$-coCartesian (note that $m_j + i -1 = n_j + i$). Applying Lemma \ref{lem:induced} yields that $e$ is $(\Ce,\Ae_{n_j+i})$-coCartesian, proving our claim.
	
We now apply the dual of Lemma \ref{lem:relative} to deduce that the restriction functor
\[
\text{Fun}'_{\Delta^{n+1}}(\text{sd}([n]),\Pi(\Ae_0,\dots,\Ae_n)) \longrightarrow \text{Fun}'{\Delta^{n+1}}(\text{sd}([n])\setminus\varnothing,\Pi(\Ae_0,\dots,\Ae_n))
\]
is a trivial Kan fibration, where
\begin{itemize}
\item $\text{Fun}'_{\Delta^{n+1}}(\text{sd}([n]),\Pi(\Ae_0,\dots,\Ae_n)) \subset \text{Fun}_{\Delta^{n+1}}(\text{sd}([n]),\Pi(\Ae_0,\dots,\Ae_n))$ is the full subcategory of functors $S$ which are $\pi$-limit diagrams,
\item $\text{Fun}'_{\Delta^{n+1}}(\text{sd}([n])\setminus\varnothing,\Pi(\Ae_0,\dots,\Ae_n))\subset \text{Fun}_{\Delta^{n+1}}(\text{sd}([n])\setminus\varnothing,\Pi(\Ae_0,\dots,\Ae_n))$ is the full subcategory of functors $L$ that admit a $\pi$-limit.
\end{itemize}
Applying \ref{lem:underlying} and the fact that $\Ce$ has finite limits, we have
\[
\text{Fun}'_{\Delta^{n+1}}(\text{sd}([n])\setminus\varnothing,\Pi(\Ae_0,\dots,\Ae_n)) = \text{Fun}_{\Delta^{n+1}}(\text{sd}([n])\setminus\varnothing,\Pi(\Ae_0,\dots,\Ae_n)).
\]
Moreover, applying the above, we have a well-defined inclusion
\[
\Se_\text{coCart} \lhook\joinrel\longrightarrow \text{Fun}'_{\Delta^{i+1}}(\text{sd}([n]), \Pi(\Ae_0,\dots,\Ae_n)).
\]
We claim that the diagram
\[\begin{tikzcd}
	{\Se_\text{coCart}} & {\text{Fun}'_{\Delta^{i+1}}(\text{sd}([n]), \Pi(\Ae_0,\dots,\Ae_n))} \\
	\Le & {\text{Fun}_{\Delta^{i+1}}(\text{sd}([n])\setminus \varnothing, \Pi(\Ae_0,\dots,\Ae_n))}
	\arrow[hook, from=1-1, to=1-2]
	\arrow[from=1-1, to=2-1]
	\arrow[from=1-2, to=2-2]
	\arrow[hook, from=2-1, to=2-2]
\end{tikzcd}\]
is a pullback square of simplicial sets.
We thus need to show that for every $L \in \Le$, the corresponding $\pi$-limit diagram $\overline{L}$ is an object of $\Se_\text{coCart}$. Specifically, we will prove that for $0 \leq i \leq n$, $\text{Cof}^{i}(q \circ \overline{L})$ has the following property: For every subdivision $n_0n_1\dots n_j$ of $[n-i]$, the edge \[
\text{Cof}^i(q\circ \overline{L})(n_0n_1\dots n_{j-1} \to n_0n_1\dots n_j)
\]
	is $(\Ce,\Ae_{n_j+i})$-coCartesian. For $i=0$ this would imply our claim. This is divided into two parts:
	\begin{enumerate}[label=(\alph*)]
		\item We show by induction starting at $i=0$ that the above condition holds for $j \geq 1$.
		\item Using (a), we show by reversed induction starting at $i=n$ that the above condition holds for $j=0$.
	\end{enumerate}
	Part (a) for $i=0$ holds by definition of $\Le$. For $1 \leq i \leq n$, we refer to the induction step proving condition (2) at the beginning of this proof.
	Part (b) for $i=n$ holds since $q \circ \overline{L}$ is a biCartesian cube and thus, $\text{Cof}^{n}(q \circ \overline{L})$ is an equivalence, and has target in $\Ae_n$. The latter claim needs to be proved: For $j \geq 1$, consider the subposet \[\text{sd}([j-1]) \cong B_j := \{n_0n_1\dots n_k \in \text{sd}([j]) \mid n_k = j\} \subset \text{sd}([j]).\]
    We have
    \begin{align*}
    \text{Cof}^n(q\circ \overline{L})(0) &= \text{Cof}\{\text{Cof}^{n-1}(q\circ \overline{L})(1 \to 01)\} \tag{$*$}\\
    &= \text{Cof}\{\text{Cof}^{n-1}(q\circ \overline L\vert_{B_{1}})\}\\
    & \hspace{2mm}\vdots\\
    &= \text{Cof}^n(q\circ \overline{L}\vert_{B_{n}}).
    \end{align*}
    Now, notice that $q\circ \overline{L}\vert_{B_{n}}$ is a diagram in $\Ae_n$, by definition of $\Le$. Thus, we have
    \[
    \text{Cof}^n(q\circ \overline{L})(0) = \text{Cof}^n(q\circ \overline{L}\vert_{B_{n}}) \in \Ae_n,
    \]
    proving the claim.
	
	Let $0 \leq i \leq n-1$ and $n_0$ a subdivision of $[n-i]$ with $n_0 \geq 1$. We have a commutative square
	\[\begin{tikzcd}[row sep=large]
	{\text{Cof}^{i}(q \circ \overline{L})(\varnothing)} & {\text{Cof}^{i}(q \circ \overline{L})(n_0)} \\
	{\text{Cof}^{i}(q \circ \overline{L})(0)} & {\text{Cof}^{i}(q \circ \overline{L})(0n_0)s}
	\arrow[from=1-1, to=1-2]
	\arrow["f"', from=1-1, to=2-1]
	\arrow["g", from=1-2, to=2-2]
	\arrow[from=2-1, to=2-2]
\end{tikzcd}\]
	and the edge of $\Ce$ corresponding via $\text{Cof}^{i+1}(q \circ \overline{L})$ to the map $\varnothing \to n_0 -1$
	is the induced map
	\[
	\text{Cof}(f) \longrightarrow \text{Cof}(g)
	\]
	which is, by the inductive hypothesis, $(\Ce, \Ae_{n_0+i})$-coCartesian. Furthermore, by (a) the lower horizontal map of the square above is also $(\Ce,\Ae_{n_0+i})$-coCartesian. By Lemma \ref{lem:induced}, this implies that the upper horizontal map is $\Ae_{n_0+i}$-coCartesian. Now, suppose that $n_0 = 0$. It follows from the inductive hypothesis that we have
	\[
	\text{Cof}^{i+1}(q \circ \overline{L})(m_0\dots m_j) \in \Ae_{m_j+i+1}
	\]
	for every subdivision $m_0\dots m_j$ of $[n-i-1]$ with $j\geq 0$. In particular, since $\text{Cof}^{i+1}(q \circ \overline{L})$ is a biCartesian cube, we have
	\[
	\text{Cof}^{i+1}(q \circ \overline{L})(\varnothing) \in \text{Span}(\Ae_{i+1},\dots,\Ae_n).
	\]
	Moreover, this object is the cofiber of the edge
    \[
    \text{Cof}^i(q \circ \overline{L})(\varnothing) \overset{e}\longrightarrow \text{Cof}^i(q \circ \overline{L})(0).
    \]This implies that $\text{Fib}(e) \in \text{Span}(\Ae_{i+1},\dots, \Ae_n) \subset {^\perp \Ae_i}$. In the proof of \ref{prop:admiss} it is shown that in this situation $e$ is $\Ae_i$-coCartesian, provided that $\text{Cof}^i(q \circ \overline{L})(0) \in \Ae_i$. This follows from a calculation analogous to $(*)$, by considering $q \circ \overline{L}$ restricted to the subposet
    \[
    \{n_0n_1\dots n_j \in \text{sd}([n]) \mid n_j = i\} \subset \text{sd}([n]). \qedhere
    \]
\end{proof}

\begin{cor}
There exists an equivalence of $\infty$-categories
\[
\Ce \longrightarrow \Le. \pushQED{\qed} \qedhere \popQED
\]
\end{cor}

\begin{rem}
Let $\Ce$ be a stable $\infty$-category equipped with a locally coCartesian semiorthogonal decomposition $(\EuScript A_0,\dots,\EuScript A_n)$. Then, in a sense not made precise in this work (see \cite[§2]{lurie-goodwillie} for a rigorous treatment), the locally coCartesian fibration $p : \chi(\Ae_0,\dots , \Ae_n) \to \Delta^n$ is classified by an \textit{oplax-commutative $n$-simplex} or in other words an \textit{oplax functor} $F : \Delta^n \to \Ce at_\infty^2$ into the $(\infty,2)$-category of $\infty$-categories. The oplax limit of $F$ in $\Ce at_\infty^2$ is equivalent to the $\infty$-category $\Le = \text{oplax-lim}(p)$. From this point of view, the above corollary says that $\Ce$ is the oplax limit of $F$ in $\Ce at_\infty^2$. Informally, this means that $\Ce$ can be reconstructed from the stable subcategories $\Ae_0,\dots , \Ae_n$ together with gluing data in the form of functors $\Ae_i \to \Ae_j$ for $i < j$ fitting into a coherent system of natural transformations.
\end{rem}

\subsection{An explicit example: Beilinson's exceptional collection}\label{subsec:beilinson}

Let $k$ be an algebraically closed field and $n \geq 0$. In \cite{beilinson}, Beilinson proved that $D^b_\text{ord}(\text{Coh}(\mathbb P^n))$ is equivalent to the derived category of $k$-linear representations of the quiver
\[
\begin{tikzcd}[column sep = large]
	0 \ar[r,"x^{(1)}_0", shift left = 2, bend left] \ar[r, shift left = 1, draw=none, "\vdots" description] \ar[r,"x^{(1)}_n"' , shift right = 2, bend right] & 1 \ar[r,"x^{(2)}_0", shift left = 2, bend left] \ar[r, shift left = 1, draw=none, "\vdots" description] \ar[r,"x^{(2)}_n"' pos= 0.51 , shift right = 2, bend right] & \cdots \ar[r,"x^{(n)}_0", shift left = 2, bend left] \ar[r, shift left = 1, draw=none, "\vdots" description] \ar[r,"x^{(n)}_n"' pos= 0.49 , shift right = 2, bend right] & n
\end{tikzcd}
\]
with commutativity relations
\[
x_i^{(t+1)}x^{(t)}_j = x_j^{(t+1)}x^{(t)}_i
\]
for $0 \leq i,j \leq n$ and $1\leq t \leq n-1$.
By applying the reconstruction theorem of §\ref{subsec:sod}.\ref{par:recons}, we obtain a homotopy-coherent version of this statement.\\

We consider $\P^n = \text{Proj}(k[x_0,\dots,x_n])$. For $i \in \Z$, we denote by
\[
\Oc_{\P^n}(i) := \widetilde{k[x_0,\dots,x_n](i)}
\]
the \textit{$n$-th twist of $\Oc_{\P^n}$}. Explicitly, for a homogeneous polynomial $f \in k[x_0,\dots,x_n]$ of positive degree, the sections $\Gamma(D(f),\Oc_{\P^n}(i))$ on the basic open $D(f)$ are given by the homogeneous elements of degree $i$ in the localization $k[x_0,\dots,x_n]_f$. The global sections $\Gamma(\P^n, \Oc_{\P^n}(i))$ are then given by the homogeneous polynomials of degree $i$ in $k[x_0,\dots,x_n]$.\\

The following is a well-known theorem due to A. I. Beilinson, see \cite{beilinson}.

\begin{thm}
We have a full exceptional collection
\[(\Oc_{\P^n}, \Oc_{\P^n}(1), \dots, \Oc_{\P^n}(n))\]
of $D^b_\text{ord}(\text{Coh}(\P^n))$. In particular, this gives a semiorthogonal decomposition
\[
(A_0,A_1, \dots, A_n)
\]
where $A_i$ is the triangulated subcategory generated by $\Oc_{\P^n}(i)$.
\qed
\end{thm}

Thus, applying Corollary \ref{prop:6}, we obtain a semiorthogonal decomposition of $D^b(\text{Coh}(\P^n))$ of the form
\[(\Ae_0,\Ae_1,\dots, \Ae_n)\]
where $\Ae_i$ is the stable subcategory generated by $\Oc_{\P^n}(i)$. Since $\mathcal O_{\P^n}(i)$ is exceptional, the objects of $\Ae_i$ are, up to equivalence, finite direct sums of shifts of $\Oc_{\P^n}(i)$. Thus, there is an equivalence
\[
\Ae_i \simeq D^b(\text{Vect}_k^\text{fin}).\]

\begin{thm}\label{prop:loc-beilinson}
The SOD $(\Ae_0,\Ae_1,\dots,\Ae_n)$ of $\Ce = D^b(\text{Coh}(\P^n))$ is locally coCartesian.
\end{thm}
\begin{proof}
	Let $0 \leq i< j \leq n$. By Lemma \ref{lem:sumcocart} below, it suffices to construct an $(\Ae_i,\Ae_j)$-coCartesian edge with source $\Oc_{\P^n}(i)$. For $t \geq 0$, let $M_t$ be the set of degree $t$ monomials in the variables $x_0,\dots,x_n$. Throughout the proof, we will use the fact that
    \[
    \mathcal Hom(\Oc_{\P^n}(i),\Oc_{\P^n}(j)) \cong \Oc_{\P^n}(j-i).
    \]
    
    Consider the map
	\[
	f : \Oc_{\P^n}(i) \longrightarrow \bigoplus_{M_{j-i}} \Oc_{\P^n}(j) =: \Fc
	\]
    given on the $p$ component for $p \in M_{j-i}$, by multiplication by $p$. Since $p$ has degree $j-i$, this is well-defined.
	We claim that it is $(\Ae_i,\Ae_j)$-coCartesian, i.e., an initial object of $\De :=\Ce_{\Oc_{\P^n}(i)/}\times_\Ce \Ae_j$. Let $g : \Oc_{\P^n}(i) \to \Gc_\bullet$ be a map. Note that we have a homotopy fiber sequence
    \[
    \text{Map}_\De(f,g) \to \text{Map}_\Ce(\Fc,\Gc_\bullet) \to \text{Map}_\Ce(\Oc_{\P^n}(i),\Gc_\bullet).
    \]
    Associated to this is a long exact sequence
    \[
    \dots \to \text{Hom}_{\text{h}\Ce}(\Oc_{\P^n}(i)[k+1],\Gc_\bullet) \to \pi_k(\text{Map}_\De(f,g)) \to \text{Hom}_{\text{h}\Ce}(\Fc[k],\Gc_\bullet) \to \text{Hom}_{\text{h}\Ce}(\Oc_{\P^n}(i)[k],\Gc_\bullet) \to \cdots
    \]
    Thus, for $f$ to be $(\Ae_i,\Ae_j)$-coCartesian it suffices to prove that the map
    \[
    \text{Hom}_{\text{h}\Ce}(\Fc[k],\Gc_\bullet) \longrightarrow \text{Hom}_{\text{h}\Ce}(\Oc_{\P^n}(i)[k],\Gc_\bullet) 
    \]
    is an isomorphism for all $k\in \Z$. We have
    \allowdisplaybreaks
    \begin{align*}
        \text{Hom}_{\text{h}\Ce}(\Oc_{\P^n}(i)[k],\Gc_\bullet) &\cong \text{Hom}_{\text{h}\Ce}(\Oc_{\P^n}(i)[k],\bigoplus_{t \in \Z} \Oc_{\P^n}(j)[t]^{\oplus m_t})\\
        &\cong\bigoplus_{t\in \Z} \text{Hom}_{\text{h}\Ce}(\Oc_{\P^n}(i)[k],\Oc_{\P^n}(j)[t])^{\oplus m_t}\\
        &\cong \bigoplus_{t\in \Z} \text{Hom}_{\text{h}\Ce}(\Oc_{\P^n}(i),\Oc_{\P^n}(j)[t-k])^{\oplus m_t}\\
        &\cong \bigoplus_{t\in \Z} \text{Ext}_{\text{h}\Ce}^{t-k}(\Oc_{\P^n}(i),\Oc_{\P^n}(j))^{\oplus m_t}\\
        &\cong \bigoplus_{t\in \Z} H^{t-k}(\P^n,\Oc_{\P^n}(j-i))^{\oplus m_t}\\
        &\cong H^0(\P^n,\Oc_{\P^n}(j-i))^{\oplus m_k}\\
        &\cong (k[x_0,\dots,x_n]_{j-i})^{\oplus m_k}.
    \end{align*}
    Similarly, we compute
    \begin{align*}
        \text{Hom}_{\text{h}\Ce}(\Fc[k],\Gc_\bullet) &\cong \text{Hom}_{\text{h}\Ce}(\bigoplus_{M_{j-i}} \Oc_{\P^n}(j)[k],\bigoplus_{t \in \Z} \Oc_{\P^n}(j)[t]^{\oplus m_t})\\
        &\cong\bigoplus_{t\in \Z} \bigoplus_{M_{j-i}} \text{Hom}_{\text{h}\Ce}( \Oc_{\P^n}(j)[k],\Oc_{\P^n}(j)[t])^{\oplus m_t}\\
        &\cong \bigoplus_{t\in \Z}\bigoplus_{M_{j-i}} H^{t-k}(\P^n,\Oc_{\P^n})^{\oplus m_t}\\
        &\cong \bigoplus_{M_{j-i}} H^0(\P^n,\Oc_{\P^n})^{\oplus m_k}\\
        &\cong (\bigoplus_{M_{j-i}} k)^{\oplus m_k}.
    \end{align*}
    Under these identifications, the map above becomes the obvious isomorphism
    \[
    (\bigoplus_{M_{j-i}} k)^{\oplus m_k} \longrightarrow (k[x_0,\dots,x_n]_{j-i})^{\oplus m_k}. \qedhere
    \]    
\end{proof}

\begin{lem}\label{lem:sumcocart}
    Let $\Ce$ be a stable $\infty$-category and $\Ae,\Be \subset \Ce$ stable subcategories. If $f : a \to b, g : a' \to b'$ are $(\Ae,\Be)$-coCartesian edges, then $f \oplus g : a \oplus a' \to b \oplus b'$ is $(\Ae,\Be)$-coCartesian.
\end{lem}
\begin{proof}
    By assumption, the maps
    \[
    \pi_1 : \Ce_{f/} \times_\EuScript C \EuScript B \longrightarrow \EuScript C_{a/} \times_\EuScript C \EuScript B
    \]
    and
    \[
    \pi_2 : \Ce_{g/} \times_\EuScript C \EuScript B \longrightarrow \EuScript C_{a'/} \times_\EuScript C \EuScript B
    \]
    are trivial Kan fibrations, and we need to show that
    \[
    \pi : \Ce_{f\oplus g/} \times_\EuScript C \EuScript B \longrightarrow \EuScript C_{a \oplus a'/} \times_\EuScript C \EuScript B
    \]
    is a trivial Kan fibration. Since this is always a left fibration (see \cite[Cor. 2.1.2.2]{lurie-htt}), it suffices to show that all fibers are contractible (see \cite[Lem. 2.1.3.4]{lurie-htt}). Note that we have equivalences
    \[
    \Ce_{f\oplus g/} \times_\EuScript C \EuScript B \simeq (\Ce_{f/} \times_\EuScript C \EuScript B) \times_\Ce (\Ce_{g/} \times_\EuScript C \EuScript B)
    \]
    and
    \[
    \EuScript C_{a \oplus a'/} \times_\EuScript C \EuScript B \simeq (\EuScript C_{a/} \times_\EuScript C \EuScript B) \times_\Ce (\EuScript C_{a'/} \times_\EuScript C \EuScript B).
    \]
    Thus, for $(s,t) \in (\EuScript C_{a/} \times_\EuScript C \EuScript B) \times_\Ce (\EuScript C_{a'/} \times_\EuScript C \EuScript B)$, we have
    \[
    \pi^{-1}(s,t) = \pi_1^{-1}(s) \times \pi_2^{-1}(t)
    \]
    which is contractible.
\end{proof}

The inner fibration $p : \chi(\Ae_0,\dots,\Ae_n) \rightarrow \Delta^n$ is thus locally coCartesian and $\Ce = D^b(\text{Coh}(\P^n))$ is identified with its oplax limit (see §\ref{groth-cons}.\ref{par:oplax}). The locally coCartesian fibration $p$ is classified via the Straightening equivalence (see Remark \ref{rem:straightening-locally}) by an oplax diagram $F : \Delta^n \to \mathbb Cat_\infty$. More concretely, we have $F(i) = \Ae_i$ and for $i < j$, the functor $F(i \to j)$ is the classifying functor for the coCartesian fibration $p\vert p^{-1}(\{i,j\})$, i.e., it sends $X \in \Ae_i$ to the target of a locally $p$-coCartesian edge $X \to Y$. By the proof of Theorem \ref{prop:loc-beilinson}, the functor $F(i \to j) : \Ae_i \to \Ae_j$ thus sends $\Oc_{\P^n}(i)$ to $\bigoplus_{M_{j-i}} \Oc_{\P^n}(j)$. Identifying $\Ae_i \simeq D^b(\text{Vect}_k^\text{fin})$ and using the fact that $\vert M_{j-i}\vert = \binom{n+j-i}{n}$, the classifying diagram for $n = 2$ is given by
\[\begin{tikzcd}[column sep = small]
	&& {D^b(\text{Vect}_k^\text{fin})} \\
	{D^b(\text{Vect}_k^\text{fin})} && {} && {D^b(\text{Vect}_k^\text{fin})}
	\arrow["\oplus 3", from=1-3, to=2-5]
	\arrow["\oplus 3", from=2-1, to=1-3]
	\arrow["\oplus 6"', from=2-1, to=2-5]
	\arrow[shorten <=3pt, shorten >=3pt, Rightarrow, "\eta" , from=2-3, to=1-3]
\end{tikzcd}
\]
where the natural transformation assigns to $X \in D^b(\text{Vect}_k^\text{fin})$, the obvious map
\[
\bigoplus_{M_2} X \longrightarrow \bigoplus_{M_1} \bigoplus_{M_1} X.
\]
An object of the oplax limit of this diagram is a choice of objects $X,Y,Z \in D^b(\text{Vect}_k^\text{fin})$, together with maps $f : X^{\oplus 3} \to Y, g : Y^{\oplus 3} \to Z, h: X^{\oplus 6} \to Z$ such that the square
\[
\begin{tikzcd}
    X^{\oplus 6} \ar[r, "\eta_X"] \ar[d, "h"] & X^{\oplus 9} \ar[d, "f^{\oplus 3}"]\\
    Z & Y^{\oplus 3} \ar[l, "g"] 
\end{tikzcd}
\]
commutes (in a homotopy-coherent way). Our theorem states that the $\infty$-category consisting of such objects is equivalent to $D^b(\text{Coh}(\P^n))$. To explain why this provides a homotopy-coherent analogue of Beilinson's statement, replace $X,Y,Z$ by ordinary vector spaces. Then, the requirement for the square above to commute strictly states exactly that $(X,Y,Z, f,g,h)$ is a $k$-linear representation of the Beilinson quiver. In particular, the commutativity relations of the quiver are encoded in the natural transformation $\eta$.


\printbibliography

\end{document}